\documentclass{article}
\usepackage{amsmath,amsfonts,amssymb,amsthm,enumerate,mathtools}
\usepackage{mathabx}
\usepackage{tabularx}
\newcolumntype{L}[1]{>{\raggedright\let\newline\\\arraybackslash\hspace{0pt}}m{#1}}
\usepackage[all]{xy}
\newtheorem{definition}{Definition}[section]
\newtheorem{proposition}[definition]{Proposition}
\newtheorem{theorem}[definition]{Theorem}
\newtheorem{lemma}[definition]{Lemma}
\newtheorem{condition}[definition]{Condition}
\newtheorem{example}[definition]{Example}
\newtheorem{corollary}[definition]{Corollary}
\newtheorem{remark}[definition]{Remark}
\newcommand{\C}{\mathbb{C}}
\newcommand{\two}{\mathbf{2}}

\newcommand{\Out}{\textnormal{Out}}
\newcommand{\Aut}{\textnormal{Aut}}
\newcommand{\Der}{\textnormal{Der}}

\newcommand{\inj}{complete}
\newcommand{\binj}{complete${}^*$}
\newcommand{\pinj}{proto-complete}
\newcommand{\scmp}{strong-complete}
\newcommand{\injness}{completeness}
\newcommand{\Injness}{Completeness}
\newcommand{\pinjness}{proto-completeness}
\newcommand{\binjness}{completeness${}^*$}
\newcommand{\scmpness}{strong-completeness}

\newcommand{\Rng}{\mathbf{Rng}}
\newcommand{\z}[1]{\mathbb{Z}/#1\mathbb{Z}}
\newcommand{\SplExt}{\mathbf{SplExt}}
\newcommand{\Act}{\mathbf{Act}}
\newcommand{\KG}{\mathbf{KGpd}}
\newcommand{\pb}[4]{#1\underset{\mathclap{\langle #3,#4\rangle}}{\times} #2}
\newcommand{\Ker}{\text{Ker}}
\newcommand{\conj}{\ensuremath{c}}
\newcounter{tmp}
\usepackage[active,
            generate=complete-theorems,
            extract-env={definition,proposition,theorem,lemma,condition%
                         ,example,corollary,remark},%
            extract-cmd=section%
                   ]{extract}
\begin{extract}

\usepackage{amsmath,amsfonts,amssymb,amsthm,enumerate,mathtools}
\usepackage{mathabx}
\usepackage{tabularx}
\newcolumntype{L}[1]{>{\raggedright\let\newline\\\arraybackslash\hspace{0pt}}m{#1}}
\usepackage[all]{xy}
\newtheorem{definition}{Definition}[section]
\newtheorem{proposition}[definition]{Proposition}
\newtheorem{theorem}[definition]{Theorem}
\newtheorem{lemma}[definition]{Lemma}
\newtheorem{condition}[definition]{Condition}
\newtheorem{example}[definition]{Example}
\newtheorem{corollary}[definition]{Corollary}
\newtheorem{remark}[definition]{Remark}
\newcommand{\C}{\mathbb{C}}
\newcommand{\two}{\mathbf{2}}

\newcommand{\Out}{\textnormal{Out}}
\newcommand{\Aut}{\textnormal{Aut}}
\newcommand{\Der}{\textnormal{Der}}

\newcommand{\inj}{complete}
\newcommand{\binj}{complete${}^*$}
\newcommand{\pinj}{proto-complete}
\newcommand{\scmp}{strong-complete}
\newcommand{\injness}{completeness}
\newcommand{\Injness}{Completeness}
\newcommand{\pinjness}{proto-completeness}
\newcommand{\binjness}{completeness${}^*$}
\newcommand{\scmpness}{strong-completeness}

\newcommand{\Rng}{\mathbf{Rng}}
\newcommand{\z}[1]{\mathbb{Z}/#1\mathbb{Z}}
\newcommand{\SplExt}{\mathbf{SplExt}}
\newcommand{\Act}{\mathbf{Act}}
\newcommand{\KG}{\mathbf{KGpd}}
\newcommand{\pb}[4]{#1\underset{\mathclap{\langle #3,#4\rangle}}{\times} #2}
\newcommand{\Ker}{\text{Ker}}
\newcommand{\conj}{\ensuremath{c}}
\newcounter{tmp}
\end{extract}
\begin{document}
\title{Complete objects in categories}
\author{James Richard Andrew Gray}
\maketitle
\begin{abstract}
 We introduce the notions of \pinj{}, \inj{}, \binj{} and \scmp{} objects
in pointed categories.
 We show under mild conditions on a pointed exact protomodular category that every \pinj{}
(respectively \inj{}) object is the
product of an abelian \pinj{} (respectively \inj{}) object and a \scmp{} object.
This together with the observation that the trivial group is
the only abelian \inj{} group recovers a theorem of Baer classifying complete
groups. In addition we generalize
several theorems about groups (subgroups) with trivial center (respectively, centralizer), and
provide a categorical explanation behind why the derivation algebra of a perfect Lie
algebra with trivial center
and the automorphism group of a non-abelian (characteristically) simple group
are \scmp.
\end{abstract}
\section{Introduction}
Recall that Carmichael \cite{CARMICHAEL:1956} called a group $G$ complete if it has trivial
center and each automorphism is \emph{inner}. For each group $G$ there is a
canonical homomorphism $\conj_G$ from $G$ to $\Aut(G)$, the automorphism group of
$G$. This homomorphism assigns to each $g$ in $G$ the inner automorphism
which sends each $x$ in $G$ to $gxg^{-1}$. It can be readily seen that a group $G$ is
complete if and only if $\conj_G$ is an isomorphism. Baer
\cite{BAER:1946}
showed that
a group $G$ is complete if and only if every normal monomorphism with
domain $G$ is a split monomorphism. We call an object in a pointed
category \inj{} if it satisfies this latter condition. \Injness{}, which corresponds
to being injective in abelian categories, has been studied in other contexts
(although not always under that name) and as explained
by B.~J.~Gardner in
\cite{GARDNER:1981}
the following  is known. Completeness corresponds to:
\begin{enumerate} 
\item having trivial center (annihilator) and each
derivation being inner in each category of Lie algebras;
\item having multiplicative identity 
in each category of associative algebras;
  \setcounter{tmp}{\arabic{enumi}}
\end{enumerate}
and follows from:
\begin{enumerate}
  \setcounter{enumi}{\arabic{tmp}}
\item the existence of a multiplicative identity 
in each category of alternative algebras and each category of
autodistributive algebras.
\end{enumerate}
Note that complete objects need not be injective nor absolute retracts, in
fact Baer showed in \cite{BAER:1946} that there are no non-trivial absolute
retracts of groups.

One of the main purposes of this paper is to introduce and study four (three main) 
notions of completeness, and to give a categorical explanation of Baer's
result mentioned above. 
In addition to \injness{} which we have already defined above we call an object $X$ \emph{\pinj{}} if
every normal monomorphism with
domain $X$ which is the kernel of a split epimorphism is a split
monomorphism, and
\emph{\scmp{}} if it satisfies the same condition except with the additional
requirement
that the normal monomorphisms in question are required to have a unique
splitting. The forth
notion is obtain by replacing normal by \emph{Bourn-normal} in the definition
of \injness{}.
Let us immediately mention that pointed protomodular categories
\cite{BOURN:1991}
in which every
object is \scmp{} turn out to be precisely what
F. Borceux and D. Bourn in \cite{BORCEUX_BOURN:2007} called coarsely
action representable.
In the pointed protomodular context we show
that \scmpness{} implies \injness{} (Proposition
\ref{proposition:implications_between_completeness}) and that every \pinj{}
(respectively \inj{}) object, satisfying certain additional conditions which
automatically hold in every such variety of universal algebras, is the product of
an abelian
\pinj{} (respectively \inj{})
object and a \scmp{} object (Theorem \ref{theorem:proto-complete_product_decomposition}).
We show that a partial converse to the previous fact holds (Proposition
\ref{proposition:product_ab_binj_scmp}). 
We give classifications of \pinjness{} and
\scmpness{} (see Theorems \ref{theorem:char_of_proto-complete} and
\ref{theorem:char_of_scmp_in_action_rep_cat}) relating to the existence of
generic split extensions in the sense of \cite{BORCEUX_JANELIDZE_KELLY:2005a},
which are closely related to Problem 6 of the open problems of
\cite{BORCEUX_JANELIDZE_KELLY:2005b}. For a group $G$ these theorems
imply: (a) $G$ is \pinj{} if and only if $\conj_G$ is a split epimorphism;
(b) $G$ is \scmp{} (= \inj{}) if and only if $\conj_G$ is an isomorphism.

Other aims include a brief study of objects with trivial center and of subobjects
with trivial centralizer, in Section \ref{section:trivial_center_and_centralizer},
and to study characteristic
monomorphism and their interaction with completeness in Section
\ref{section:characteristic monos}.
The main results of Section
\ref{section:trivial_center_and_centralizer},
applied to the category of groups, recover the following known facts about groups:
\begin{enumerate}[(i)]
\item If $G$ has trivial center, then $\conj_G : G\to \Aut(G)$ is 
a normal monomorphism with trivial centralizer 
(Proposition
\ref{proposition:X_centerless_c_X_centralizer_free});
\item If $n:N\to G$ is a normal monomorphism with trivial centralizer,
then each automorphism of $N$ admits at most one extension to $G$
(Proposition
\ref{proposition:m_trivial_centralizer=>q_1_mono}).
\end{enumerate}
In Section
\ref{section:characteristic monos}
we provide a common categorical explanation behind why the derivation
algebra of a perfect Lie algebra with trivial center, and the automorphism
group of a (characteristically) simple group are (strong) complete. This
explanation depends on
several facts including: (a) two new characterizations of characteristic
monomorphisms with domain satisfying certain conditions (see Theorems
\ref{theorem:char_of_char}
and
\ref{theorem:char_of_centerless_char});
(b) Theorem
\ref{theorem:one_step},
which generalizes the following fact for a group $G$: the homomorphism 
$\conj_G: G\to \Aut(G)$ is a characteristic monomorphism if and only if $G$
has trivial center and $\Aut(G)$ is (strong) complete.
%The paper is organized as follows. Section 2 is devoted to recalling preliminary
%facts which will be needed in the subsequent sections. In Section 3 we study the
%condition, forming part of the definition of a complete object given in the next
%section, of when an object has only the trivial central subobject, as well as
%other closely related conditions.  In final section we study how give several
%reformulations of when a subobjects is characteristic in the sense of
%A.~S.~Cigoli and A.~Montoli \cite{} and ......
%%Let $\C$ be a finitely complete category and let $\Gpd(\C)$ be the category of
%groupoids in $\C$ and let $\Pt(\C)$ be the category of split epimorphisms and
%let $U : \Gpd(\C) \to \Pt(\C)$ be the forgetful functor sending a groupoid to
%the split epimorphism consisting of the domain morphism and its section (the
%morphism which sends each object to its section. A simple but important (FIX
%ME: judgemental) result of D. Bourn is that this functor  $U : \Gpd(\C) \to
%\Pt(\C)$ is monadic. Let us recall why this is so. For each split epimorphism
%\[ \xymatrix{ A \ar@<0.5ex>[r]^{\alpha} & B \ar@<0.5ex>[l]^{\beta} } \] the
%kernel pair of $\alpha$ are the projections of an equivalence relation. The
%left adjoint of $U$ sends a split epimorphism to this equivalence relation
%considered as a groupoid. Trivially the functor $U$ reflects isomorphisms.
\section{Preliminaries}
\label{section: preliminaries}
In this section we recall preliminary definitions and introduce some notation.

Let $\C$ be a pointed category with finite limits. We will write $0$ for both
the zero object and for a zero morphism between objects.
For objects $A$ and $B$
we will write $A\times B$ for the product of $A$ and $B$, and write 
$\pi_1$ and $\pi_2$ for the first and second product projections,
respectively. For a pair of morphisms $f:W\to A$ and $g:W\to B$ we will write
$\langle f,g\rangle : W \to A\times B$ for the unique morphism with
$\pi_1 \langle f,g\rangle = f$ and $\pi_2 \langle f,g\rangle = g$. For objects 
$A$ and $B$ we write
$A+B$ for the coproduct (when it exists), and write $\iota_1$ and $\iota_2$ for the first and
second coproduct inclusion. For a pair of morphisms $u:A\to Z$ an $v:B\to Z$
we write $[u,v]:A+B\to Z$ for the unique morphism with $[u,v]\iota_1=u$ and
$[u,v]\iota_2=v$. 

The category $\C$ is called 
unital \cite{BORCEUX_BOURN:2004} if for each pair of objects $A$ and $B$ the morphisms $\langle
1,0\rangle: A \to A\times B$ and $\langle 0,1\rangle : B\to A\times B$ are
jointly strongly epimorphic.  When these morphism are only jointly epimorphic
$\C$ is called weakly unital \cite{MARTINS-FERREIRA:2008a}. A pair of
morphisms $f:A\to X$ and $g: B\to X$ in a (weakly) unital category are said to
(Huq)-commute if there exists a (unique) morphism $\varphi : A\times B\to X$ making
the diagram
\[ \xymatrix{
A
\ar[r]^-{\langle 1,0\rangle}\ar@/_3ex/[dr]_{f} & A\times B \ar[d]^{\varphi} &
B\ar[l]_-{\langle 0,1\rangle}\ar@/^3ex/[dl]^{g}\\ & X &
}
\]
commute.  
Let us recall some well-known facts and definitions related to the commutes
relation (see e.g. \cite{BORCEUX_BOURN:2004}, and the references there for
the unital context, and \cite{GRAY:2010a,GRAY:2012b} for the weakly unital context.)
\begin{lemma}
\label{lemma:standard_properties_of_commutes_relation}
Let $\C$ be a weakly unital category and let $e:S\to A$, $f: A\to X$, $g: B\to
X$, $f':A'\to X'$, $g':B'\to X'$ and $h: X\to Y$ be morphisms in $\C$, then
\begin{enumerate}[(i)] \item $f$ and $g$ commute if and only if $g$ and $f$
commute; \item if $f$ and $g$ commute, then $fe$ and $g$ commute; \item if $f$
and $g$ commute, then $hf$ and $hg$ commute; \item $f\times f'$ and $g\times g'$
commute if and only if $f$ and $g$, and $f'$ and $g'$ commute.  
\end{enumerate}
Moreover, the converse of (ii) holds when $e$ is a pullback stable regular
epimorphism, while the converse of (iii) holds when $\C$ is unital and $h$ is a
monomorphism.  
\end{lemma}
 
\begin{definition} Let $\C$ be a weakly unital
category. The centralizer of a morphism $f:A\to B$ is the terminal object in the
category of morphisms commuting with $f$.  
\end{definition}
 We will write
$z_{f}: Z_X(A,f)\to X$ for the centralizer of $f$ when it exits. Note that $z_f$
is always a monomorphism. When $f=1_X$ the centralizer of $f$
is called the center of $X$ and will be denoted $z_X:Z(X)\to X$. 

A split extension in $\C$ is a diagram
\begin{equation}
\label{diag:split_ext}
\vcenter{
\xymatrix{
X \ar[r]^{\kappa} & A\ar@<0.5ex>[r]^{\alpha} & B\ar@<0.5ex>[l]^{\beta}
}
}
\end{equation}
in $\C$
where $\kappa$ is the kernel of $\alpha$ and $\alpha\beta=1_B$. A morphism of
split extensions in $\C$ is a diagram
\begin{equation}
\label{diag:mor_split_ext}
\vcenter{
\xymatrix{
X \ar[r]^{\kappa}\ar[d]_{u} & A\ar[d]^{v}\ar@<0.5ex>[r]^{\alpha} & 
B\ar[d]^{w} \ar@<0.5ex>[l]^{\beta}\\
X' \ar[r]^{\kappa'} & A'\ar@<0.5ex>[r]^{\alpha'} & B'\ar@<0.5ex>[l]^{\beta'}
}
}
\end{equation}
in $\C$
where the top and bottom rows are split extensions (the domain and codomain,
 respectively), such that $\kappa' u=v\kappa$, $\beta' w= v\beta$ and
$\alpha' v = w \alpha$. Let us denote by $\SplExt(\C)$ the category of split
extensions, and by $K$ the functor sending
\eqref{diag:split_ext}
and
\eqref{diag:mor_split_ext}
to $X$ and $u$, respectively.
Let us write $\KG(\C)$ for the category with objects 8-tuples
$(X,G_0,G_1,d,c,e,m,k)$ consisting of objects and morphisms such that the
diagram on the left
\begin{equation}
\label{diagram:groupoid_with_underlying_split_extensions}
\vcenter{
\xymatrix{
 \pb{G_1}{G_1}{d}{c} \ar[r]^-{m} & G_1 \ar@<0.75ex>[r]^{d} 
\ar@<-0.75ex>[r]_{c} & G_0 \ar[l]|{e} &
X \ar[r]^{k} & G_1 \ar@<0.5ex>[r]^{d} & G_0 \ar@<0.5ex>[l]^{e}
}
}
\end{equation}
is a groupoid and the diagram on the right is a split extension. A morphism 
$(X,G_0,G_1,d,c,e,m,k) \to (X',G'_0,G'_1,d',c',e',m',k')$ is a triple 
$(u,v,w)$ where 
$u:X\to X'$, $v:G_1\to G'_1$ and $w: G_0\to G'_0$ are morphisms in $\C$ such
that the diagram on the left is a functor
\[
\xymatrix@R=3ex{
 \pb{G_1}{G_1}{d}{c} \ar[r]^-{m}\ar[d]^{v\times v} & G_1\ar[d]^{v} 
\ar@<0.75ex>[r]^{d} \ar@<-0.75ex>[r]_{c} & G_0\ar[d]^{w} \ar[l]|{e} &\\
 \pb{G'_1}{G'_1}{d'}{c'}  \ar[r]^-{m'} & G'_1 \ar@<0.75ex>[r]^{d'} 
\ar@<-0.75ex>[r]_{c'} & G'_0 \ar[l]|{e'} &
}
\xymatrix{
X \ar[r]^{k}\ar[d]^{u} & G_1\ar[d]^{v} \ar@<0.5ex>[r]^{d} & G_0\ar[d]^{w} 
\ar@<0.5ex>[l]^{e}\\
X \ar[r]^{k} & G_1 \ar@<0.5ex>[r]^{d} & G_0 \ar@<0.5ex>[l]^{e}
}
\]
and the diagram on the right is a morphism of split extensions. Note that the 
forgetful functor $U:\KG(\C) \to \SplExt(\C)$ is monadic since it is essentially
the same as the forgetful functor from the category of internal groupoids in $\C$
to the category of split epimorphisms in $\C$ shown to be monadic by D.{}
 Bourn in \cite{BOURN:1987} (see also the discussion above Theorem 4.2 of
\cite{GRAY:2017}).

A pointed category $\C$ can be equivalently defined to be
(Bourn)-protomodular \cite{BOURN:1991}
if the split short five lemma holds, that is, for each morphism of split
extensions \eqref{diag:mor_split_ext}
if $u$ and $w$ are isomorphisms, then $v$ is an isomorphism.
A category $\C$ is semi-abelian in the sense of
G.{} Janelidze, L.{} Marki and W.{} Tholen
\cite{JANELIDZE_MARKI_THOLEN:2002}
if it is pointed, (Barr)-exact
\cite{BARR:1971},
protomodular and has binary coproducts. 
Following, F.{} Borceux, G.{} Janelidze, G.{} M.{} Kelly,
in
\cite{BORCEUX_JANELIDZE_KELLY:2005a}
we define a generic
split extension with kernel $X$ to be a terminal object in a fiber
$K^{-1}(X)$ of $K : \SplExt(\C)\to \C$. We denote such a generic split
extension as follows:
\begin{equation}
\label{diagram:generic_split_extension_with_kernel_X}
\vcenter{
\xymatrix{
X \ar[r]^-{k} & [X]\ltimes X\ar@<0.5ex>[r]^-{p_1} & [X].\ar@<0.5ex>[l]^-{i}
}
}
\end{equation}
A semi-abelian category is called action representable if each object
admints a generic split extension with kernel $X$.
Examples of action representable categories such include the category
of groups where $[X]=\Aut(X)$
is the automorphism group of $X$,
and the category of Lie algebras over a commutative ring $R$ where
$[X]=\Der(X)$ is the Lie algebra of derivations of $X$
(see \cite{BORCEUX_JANELIDZE_KELLY:2005a}). Other examples can be found
in \cite{BORCEUX_JANELIDZE_KELLY:2005b}, \cite{BORCEUX_BOURN_JOHNSTONE:2006},
\cite{BORCEUX_CLEMENTINO_MONTOLI:2014}, \cite{GRAY:2010}, \cite{GRAY:2013b}.
For a pointed
protomodular category $\C$ and for an object $X$ in $\C$ such that the
generic split extension with kernel $X$ exists, the object $[X]$ is called
the split extension classifier for $X$ and has the following universal property:
For each split extension \eqref{diag:split_ext} there exist a unique
morphism $v: B\to X$ such that there exist $u: A\to [X]\ltimes X$ making
the diagram a morphism of split extensions
\begin{equation}
\label{diagram:morphims_of_generic_split_extension_with_kernel_X}
\vcenter{
\xymatrix{
 X \ar[r]^{\kappa}\ar@{=}[d] & A\ar@<0.5ex>[r]^{\alpha}\ar[d]^{u} & B\ar@<0.5ex>[l]^{\beta}\ar[d]^{v}\\
X \ar[r]^-{k} & [X]\ltimes X\ar@<0.5ex>[r]^-{p_1} & [X]\ar@<0.5ex>[l]^-{i}
}
}
\end{equation}
(see the last theorem of Section 6 of \cite{BORCEUX_JANELIDZE_KELLY:2005a}.)

D.{} Bourn and G. Janelidze call a split
extension with kernel $X$ faithful
\cite{BOURN_JANELIDZE:2009}
when
there is at most one morphism to it from any split extension in $K^{-1}(X)$.
A semi-abelian category is called action accessible
\cite{BOURN_JANELIDZE:2009} if for each $X$ in $\C$ there is a morphism
from each split extension to a faithful one in $K^{-1}(X)$. Examples of
action accessible categories include the category of not-nessesarily unital
rings, associative algebras and more generally categories of interest in
the sense of G. Orzech \cite{ORZECH:1972} (see \cite{BOURN_JANELIDZE:2009}
and \cite{MONTOLI:2010}.)

Throughout the paper we denote by $\two$ the category with objects $0$ and $1$,
and with one non-identity morphism $0\to 1$. We will identify the functor category
$\C^\two$ with the category
of morphism of $\C$, and its objects will be written as triples $(X,Z,f)$ 
where $X$ and $Z$ are objects and $f:X\to Z$ is a morphism in $\C$. A
morphism $(X,Z,f)\to (X',Z',f')$ with be written as a pair $(u,v)$ where
$u : X\to X'$ and $v:Z'\to Z'$ are morphisms in $\C$ with $fu=vf'$.

Recalling from \cite{GRAY:2013b}, in the \emph{semi-abelian context}, or from
\cite{GRAY:2017} for the general pointed context, that each generic split
extension in
$\mathbb{C}^\two$ has codomain a generic split extension in $\C$,
we will denote a generic split extensions of a morphism $f:X\to Z$
in $\mathbb{C}^\two$ as follows:
\begin{equation}\label{generic_split_ext_in_mor_cat}
\vcenter{
\xymatrix{
X\ar[d]^{f} \ar[r]^-{k} & 
[X,Z,f]\ltimes X\ar@<0.5ex>[r]^-{p_1}\ar[d]^{q_2\ltimes f} & 
[X,Z,f]\ar@<0.5ex>[l]^-{i}\ar[d]^{q_2}\\
Z \ar[r]^-{k} & [Z]\ltimes Z\ar@<0.5ex>[r]^-{p_1} & [Z].\ar@<0.5ex>[l]^-{i}
}
}
\end{equation}

Note that according to the universal property of generic split extensions 
there is also a unique morphism
\[
\xymatrix{
X\ar@{=}[d] \ar[r]^-{k} & 
[X,Z,f]\ltimes X\ar@<0.5ex>[r]^-{p_1}\ar[d]^{q_1\ltimes 1} & 
[X,Z,f]\ar@<0.5ex>[l]^-{i}\ar[d]^{q_1}\\
X \ar[r]^-{k} & [X]\ltimes X\ar@<0.5ex>[r]^-{p_1} & [X].\ar@<0.5ex>[l]^-{i}
}
\]

When $X$ admits a generic split extension, writing
$(R,r_1,r_2)$ for the kernel pair of $p_1$, it turns out that the unique morphism
\[
\xymatrix{ X \ar[r]^-{\langle 0,k\rangle} \ar@{=}[d] & R
\ar@<0.5ex>[r]^-{r_1}\ar[d]^{s}  & [X]\ltimes X \ar@<0.5ex>[l]^{\langle 1,1\rangle}\ar[d]^{p_2}\\ X
\ar[r]_-{k} & [X]\ltimes X\ar@<0.5ex>[r]^-{p_1}& [X]\ar@<0.5ex>[l]^-{i}
}
\]
is an algebra structure for the monad induced by the adjunction
where $U:\KG(\C)\to \SplExt(\C)$ is a right adjoint, and
hence determines
a groupoid structure on \eqref{diagram:generic_split_extension_with_kernel_X}.
Putting all the structure together we obtain an object
$(X,[X]\ltimes X,[X],p_1,p_2,i,m,k)$ in $\KG(\C)$ which turns out
to be terminal in the category of $X$-groupoids, that is, the fiber 
$(KU)^{-1}(X)$
(see Proposition 5.1 of \cite{BORCEUX_BOURN:2007} where this is proved
directly in the pointed protomodular context, or combine Proposition 2.26
and Theorem 4.1 via the remarks before Theorem 4.2 of \cite{GRAY:2017},
for the pointed finitely complete context).
Writing $\nabla(X)$ for the object in $\KG(\C)$ with underlying groupoid the
indiscrete groupoid on $X$ and with $\langle 0,1\rangle: X\to X\times X$ as
kernel of its domain morphism 
$\pi_1:X\times X\to X$, we call the morphism 
$\conj_X: X\to [X]$
which forms part of the unique morphism of $X$-groupoids
$\nabla(X)\to (X,[X]\ltimes X,[X],p_1,p_2,i,m,k)$
the conjugation
morphism of $X$ (see Theorem 3.1 of \cite{BORCEUX_BOURN:2007}). In particular
this means it forms part of the unique
morphism
\begin{equation}
\label{diagram:conjugation_morphism}
\vcenter{
\xymatrix{ 
X\ar@{=}[d] \ar[r]^-{\langle 0,1\rangle} & X\times
X\ar[d]^{\varphi} \ar@<0.5ex>[r]^-{\pi_1} & X\ar@<0.5ex>[l]^-{\langle
1,1\rangle}\ar[d]^{\conj_X}\\ X \ar[r]^-{k} & [X]\ltimes X \ar@<0.5ex>[r]^-{p_1} &
[X],\ar@<0.5ex>[l]^-{i} }
}
\end{equation}
and that $p_2k=c_X$. The kernel of $\conj_X$ turns out to be $z_X: Z(X)\to X$
the center of $X$ (see Proposition 5.2 of \cite{BOURN_JANELIDZE:2009} and 
Proposition 4.1 of \cite{BORCEUX_BOURN:2007} which is closely related).

For a semi-abelian category $\C$ let
us briefly recall
the equivalence between $\SplExt(\C)$ and $\Act(\C)$ the category of
\emph{internal object actions} from
\cite{BORCEUX_JANELIDZE_KELLY:2005a}.
The objects of $\Act(\C)$ are triples $(B,X,\zeta)$ where
$B$ and $X$ and are objects in $\C$ and $\zeta : B\flat X\to X$ is a morphism making $(X,\zeta)$
an algebra over a certain monad whose
object functor sends $X$ to $B\flat X$ where $\kappa_{B,X}:B\flat X\to B+X$ is the kernel of $[1,0]:B+X\to B$.
Given morphisms $g:B\to B'$ and $f:X\to X'$ the unique morphism $g\flat f : B\flat X \to B'\flat X'$
making the diagram
\[
\xymatrix{
B\flat X\ar[d]_{g\flat f}\ar[r]^{\kappa_{B,X}} & B+X\ar[d]^{g+f} \ar@<0.5ex>[r]^-{[1,0]} &
B\ar[d]^{g}\ar@<0.5ex>[l]^-{\iota_1}\\
B'\flat X' ]\ar[r]_-{\kappa_{B',X'}} & B'+X'\ar@<0.5ex>[r]^-{[1,0]} & B'\ar@<0.5ex>[l]^-{\iota_1}
}
\]
a morphism of split extensions, makes $-\flat * : \C\times \C\to \C$ a functor. A morphism 
$(B,X,\zeta) \to (B',X',\zeta')$ in $\Act(\C)$ is a pair $(g,f)$, where $g:B\to B'$ and $f:X\to X'$
are morphisms in $\C$ such that $f \zeta = \zeta' (g\flat f)$.

Given a split extension \eqref{diag:split_ext} the unique morphism $\zeta : B\flat X\to X$
making the diagram
\[
\xymatrix{
B\flat X \ar[r]^-{\kappa_{B,X}}\ar@{-->}[d]_{\zeta} & B+X \ar@<0.5ex>[r]^-{[1,0]}\ar[d]^{[\beta,\kappa]}
& B\ar@<0.5ex>[l]^-{\iota_1}\ar@{=}[d]\\
X\ar[r]_{\kappa} & A \ar@<0.5ex>[r]^{\alpha} & B \ar@<0.5ex>[l]^{\beta}
}
\]
a morphism of split extensions, makes $(B,X,\zeta)$ an object in $\Act(\C)$ and this
assignment determines the object map of a functor $\SplExt(\C) \to \Act(\C)$ which
is an equivalence of categories. Using this equivalence of categories G.{} Janelidze
produced an equivalence of categories between internal categories in $\C$
and internal crossed modules in $\C$ \cite{JANELIDZE:2003}. Let us write
$\gamma_X : X\flat X\to X$ for the action corresponding to the domain of morphism of
split extensions \eqref{diagram:conjugation_morphism}. When \emph{star multiplicative}
graphs coincide with \emph{multiplicative} graphs (which N.{} Martins-Ferreira and T.{} 
Van der Linden showed, in
\cite{MARTINS-FERREIRA_VAN_DER_LINDEN:2012},
 happens exactly when the Huq \cite{HUQ:1968} and
Smith-Pedicchio
\cite{PEDICCHIO:1995}
commutators coincide) an internal crossed module can be equivalently
defined as a triple $(B,X,\zeta,f)$ where $(B,X,\zeta)$ is an object in $\Act(\C)$ and
$f:X\to B$ is a morphism in $\C$ such that $(f,1_X) : (X,X,\gamma_X) \to (B,X,\zeta)$
and $(1_B,f) : (B,X,\zeta) \to (B,B,\gamma_B)$ are morphisms in $\Act(\C)$ (see 
\cite{JANELIDZE:2003} noting that $\gamma_X = [1,1]\kappa_{X,X}$ and that
$[1,f]\kappa_{B,X}=\gamma_B(1_B\flat f)$).
When $\C$ is a semi-abelian category and $\tau$ is the action corresponding to
split extension \eqref{diagram:generic_split_extension_with_kernel_X} then,
via the equivalence between internal categories and internal crossed modules, it follows
that the quadruple $([X],X,\tau,\conj_X)$ is the
terminal object in the category internal crossed modules with domain
of the underlying morphism the object $X$.

A monomorphism $m:S\to X$ in a finitely complete category $\C$ is
Bourn-normal \cite{BOURN:2000a}  to an equivalence relation
$r_1,r_2:R\to X$ if there exists a morphism $\tilde m:S\times S \to R$
such that (either and hence both of) the squares of the diagram on left,
or equivalently the diagram on the right
\[
\xymatrix{
S\ar[d]_{m}&S\times S\ar[d]^{\tilde m}\ar[r]^-{\pi_2}\ar[l]_-{\pi_1}&S\ar[d]^{m}\\
X&R\ar[r]_{r_2}\ar[l]^{r_1}&X
}
\xymatrix{
S\times S \ar[r]^{\tilde m}\ar[d]_{1\times m} & R
\ar[d]^{\langle r_1,r_2\rangle}\\
S\times X \ar[r]_{m\times 1} & X\times X
}
\]
are pullbacks.
When $\C$ is pointed, such a morphism $\tilde m$
exists as soon as there is a morphism $k:S\to R$ making the lower left hand
square in the diagram
\begin{equation}
\label{diagram:bourn_normal}
\vcenter{
\xymatrix@!@C=0ex@R=0ex{ & S\ar[rr]^{\langle
0,1\rangle}\ar@{=}[dl]\ar@/^3ex/[ddl]^(0.3){m} && S\times S \ar[dl]_{\tilde
m} \ar@<0.5ex>[rr]^{\pi_1}\ar@/^3ex/[ddl]^(0.2){m\times m} && S
\ar[dl]_{m}\ar@<0.5ex>[ll]^{\langle 1,1\rangle}\ar@/^3ex/[ddl]^{m}\\ S
\ar[rr]^(0.7){k}\ar[d]_{m} && R \ar[d]_{\langle r_1,r_2\rangle}
\ar@<0.5ex>[rr]^(0.7){r_1} && X\ar@{=}[d] \ar@<0.5ex>[ll]^(0.3){s}\\ X
\ar[rr]^{\langle 0,1\rangle} && X\times X \ar@<0.5ex>[rr]^{\pi_1} && X
\ar@<0.5ex>[ll]^{\langle 1,1\rangle} } 
}
\end{equation}
a pullback. Note that, writing $s$ for the unique morphism such that
$r_1s=1_X=r_2s$, this means that the entire diagram consists of morphisms of
split extensions. Note also that this means that for an equivalence relation
$(R,r_1,r_2)$ if $k$ is the kernel of $r_1$, then the composite $r_2k$ is 
Bourn-normal to $(R,r_1,r_2)$. We call a morphism $m$ a Bourn-normal
monomorphism as soon as there exists an equivalence relation to which it is
Bourn-normal.
\section{Objects with trivial centers and morphisms with trivial centralizers}
\label{section:trivial_center_and_centralizer}
We will say that:
\begin{enumerate}[(i)] \item An object $X$ in a weakly unital category has
trivial center if $Z(X)=0$; \item A morphism $f:A\to X$ in a weakly unital
category has trivial centralizer if $Z_X(A,f)=0$.  
\end{enumerate}
 In this
section we study objects and morphisms with these properties.
The following lemma is closely related to Proposition 3.18 of 
\cite{BOURN:2001}.
\begin{lemma}
\label{lemma:char_of_centerless}
Let $\C$ be a (weakly) unital category. An object $X$ in $\C$ has trivial center
if and only if for each object $Y$ the morphism
$\langle 1,0\rangle : X \to X\times Y$ has a unique section.
\end{lemma}
\begin{proof}
The claim follows by noting that for each commutative diagram
\[
\xymatrix{
X \ar[r]^-{\langle 1,0 \rangle}\ar[dr]_{1_X}& X\times Y \ar[d]^-{\varphi} &
Y \ar[l]_-{\langle 0,1\rangle}\ar[dl]^{f}\\
& X &
}
\]
$f$ is a zero morphism if and only if $\varphi=\pi_2$.
\end{proof}

Since for any morphisms $e: S\to A$, $f: A\to X$ and $g:B\to X$ in a weakly
unital category, if $f$ and $g$ commute, then so do $fe$ and $g$, we obtain:
\begin{proposition} Let $e:S\to A$ and $f:A\to X$ be morphisms in a weakly
unital
category $\C$.  If $fe$ has trivial centralizer, then so does $f$. In
particular when $A=X$ and $f=1_X$ this means that if $e$ has trivial
centralizer, then $X$ has trivial center.  
\end{proposition}
\begin{proposition}
\label{proposition:center_of_products} Let $\C$ be a
weakly
unital category and let $(X,(\pi_i : X\to X_i)_{i \in I})$ be the product of a
family of objects $(X_i)_{i\in I}$ in $\C$.  
\begin{enumerate}[(i)]
\item A morphism $f:A\to X$ commutes with $1_X$ if and only if
$\pi_i f: A \to X_i$ commutes with $1_{X_i}$ 
\item If the center $z_{X_i}:Z(X_i)\to X_i$ of each
$X_i$ exists and the product of family $(Z(X_i)_{i\in I})$ exists, then
the center of $X$ exists and is the product of the family of morphisms
$(z_{X_i}:Z(X_i)\to X_i)_{i\in I}$;
\item if the center of $X$ exists, then the center of each $X_i$ exist and
 their product exists and is the center of $X$.
\end{enumerate}
\end{proposition}
\begin{proof} Since (ii) is a straight forward consequence of (i) we prove
(i) and (iii). To prove (i) let $f:A\to X$ be a morphism in $\C$.  Since each $\pi_i$ being a split
epimorphism is a pullback stable regular epimorphism it follows by Proposition
3.14 in \cite{GRAY:2012b} that $\pi_i f$ commutes with $\pi_i$ if and only if
$\pi_i f$ commutes with $1_{X_i}$. Therefore we need only show that $f$ 
commutes
with $1_X$ if and only if $\pi_i f$ commutes with $\pi_i$. However, this 
follows
from the fact that the existence of a morphism $\varphi : A\times X \to X$
such
that $\varphi \langle 1,0\rangle = f$ and $\varphi \langle 0,1\rangle = 1_X$ 
is
equivalent to the existence of a family of morphisms 
$(\varphi_i : A\times X \to X_i)_{i\in I}$
such that for each $i$ in $I$, 
$\varphi_i \langle 1,0\rangle = \pi_i f$
and 
 $\varphi_i \langle 0,1\rangle = \pi_i$. To prove (iii) let $z:Z\to X$ be 
 the center of $X$. For each $i$ in $I$ let $\lambda_i:X_i\to X$ be the
 unique morphism with $\pi_j\lambda_i$ identity if $i=j$ and 
 zero otherwise, and let $z_i : Z\to X$ be the morphism obtained by pullback
 as displayed in the square of the diagram
 \[
  \xymatrix{
   Z\ar@/_2ex/[ddr]_{\pi_i z}\ar@{-->}[dr]^{\rho_i} &&\\
   & Z_i \ar[r]^{\eta_i} \ar[d]_{z_i} & Z\ar[d]^{z} \\
   & X_i \ar[r]^{\lambda_i} & X.\\
  }
 \]
 We will show that $z_i$ is the center of $X_i$.
Suppose $f:A\to X_i$ a central morphism. Since as easily follows from (i)
the morphism
 $\lambda_i f : A\to X$ is central, it follows that there exists
 a unique morphism $\tilde f: A \to Z$ such that
 $z\tilde f = \lambda_i f$ and hence a unique morphism $\bar f : A\to Z_i$
 such that $z_i \bar f = f$ and $\eta_i \bar f=\tilde f$. Since
 $z_i$ is a monomorphism this is sufficient to show that it is the center of $X_i$.
 Since $z$ commutes with $1_X$ it follows that $\pi_i z$ commutes with $\pi_i$
 and hence, by Proposition 3.14 in \cite{GRAY:2012b}, that $\pi_i z$ is central.
 This means that there is a unique morphism $\rho_i:Z\to Z_i$ such that
 $z_i\rho_i = \pi_iz$.
We will show that the family $(Z,(\rho_i: Z\to Z_i)_{i \in I})$ is a product.
 Suppose
 $(\alpha_i : A \to Z_i)_{i\in I}$ is family of morphisms. By composition we obtain
 a family $(z_i \alpha_i :A\to X_i)_{i\in I}$ and hence a unique morphism $\alpha : A \to X$
 such $\pi_i \alpha = z_i \alpha_i$ for each $i$ in $I$. Since by applying (i)
 we see that $\alpha$ is central it follows that there exists a unique morphism
 $\bar \alpha : A\to Z$ such that $z \bar \alpha =\alpha$. However, since for
 each $i$ in $I$ the morphism $z_i$ is a monomorphism and
 $z_i \rho_i \bar \alpha = \pi_i z \bar \alpha = \pi_i \alpha = z_i \alpha_i$
 it follows that $\rho_i \bar \alpha = \alpha_i$.
 The proof of claim is completed by noting that the family $(\rho_i)_{i\in I}$ is jointly
 monomorphic. To see why note that $(\pi_iz)_{i \in I}$ is jointly monomorphic
 (since $z$ is a monomorphism and $(\pi_i)_{i\in I}$ is jointly monomorphic) and 
 for each $i$ in $I$, $z_i$ is a monomorphism and $z_i\rho_i=\pi_iz$.
\end{proof}
 As a corollary we obtain: 
\begin{corollary}
\label{corollary:factors_of_a_product_centerless_<=>_product_centerless}
Let $\C$ be a weakly
unital category and let $(X,(\pi_i : X\to X_i)_{i \in I})$ be the product of a
family of objects $(X_i)_{i\in I}$ in $\C$.  The object $X$ has trivial center
if and only if for each $i$ in $I$ the object $X_i$ has trivial center.
\end{corollary}
Recall that an object in the category of $X$-groupoids is called faithful
\cite{BOURN_JANELIDZE:2009}
if it admits at most one morphism from each object in the category
of $X$-groupoids. Recall also that a faithful $X$-groupoid has underlying
split extension faithful (see Lemma 3.2 \cite{BOURN_JANELIDZE:2009}), and that
an internal reflexive graph in a protomodular (more generally Mal'tsev)
category admits at most one groupoid structure. We will therefore, in 
the protomodular (more generally Mal'tsev) context, consider groupoids as
special kinds of reflexive graphs.
\begin{lemma} Let $\C$ be a pointed protomodular category. If $X$ has trivial
center and 
\begin{equation}
\label{diag:trivial_centralizer}
\vcenter{
\xymatrix@C=7ex{ X \ar@{=}[d]\ar[r]^-{\langle 0,1\rangle} &
X\times X\ar[d]^{\varphi} \ar@<1ex>[r]^-{\pi_1}\ar@<-1ex>[r]_-{\pi_2} &
X\ar[d]^{c}\ar[l]|-{\langle 1,1\rangle}\\ X \ar[r]^{\kappa} & A
\ar@<1ex>[r]^{\alpha_1} \ar@<-1ex>[r]_{\alpha_2} & B \ar[l]|-{\beta} } 
}
\end{equation}
is a morphism of $X$-groupoids, then $(A,\alpha_1,\alpha_2)$ is an
equivalence relation and $c$ is a Bourn-normal
monomorphism normal to $(A,\alpha_1,\alpha_2)$.
When, in addition, the $X$-groupoid at the
bottom of \eqref{diag:trivial_centralizer} is faithful, then $c$ has trivial centralizer.
\end{lemma}
 
\begin{proof} Since
 according to Lemma 2.4 of \cite{CIGOLI_MANTOVANI:2012} (see also Proposition 5.2
 of \cite{BOURN_JANELIDZE:2009}) the morphism
$\ker(c) : \Ker(c) \to X$ is a subobject of $X$ such that
$\langle 0,1\rangle$ and $\langle 1,1\rangle \ker(c)$ commute, it easily
 follows that $\ker(c)$ and $1_X$ commute (see Observation 5.3 of 
 \cite{BOURN_JANELIDZE:2009} or Corollary 2.6 of \cite{CIGOLI_MANTOVANI:2012}).
 This means that
$\ker(c)=0$ and hence, by protomodularity, $c$ is a monomorphism 
\cite{BOURN:2000}.  
Since
$c= c\pi_2\langle 0,1\rangle = \alpha_2 \varphi \langle 0,1\rangle = \alpha_2
\kappa$ it follows that the diagram 
\begin{equation}
\label{diagram:1}
\vcenter{
\xymatrix@C=7ex{ X \ar[r]^{\kappa}\ar[d]_{c} & A
\ar@<0.5ex>[r]^{\alpha_1}\ar[d]^{\langle \alpha_1,\alpha_2\rangle}  & B
\ar@<0.5ex>[l]^-{\beta}\ar@{=}[d]\\ B \ar[r]^-{\langle 0,1\rangle} & B\times B
\ar@<0.5ex>[r]^-{\pi_1} & B \ar@<0.5ex>[l]^-{\langle 1,1\rangle} } }
\end{equation}
 commutes. Therefore, since $\ker(\langle \alpha_1,\alpha_2\rangle)=\ker(c)=0$ 
 it follows by protomodularity that
$\langle \alpha_1,\alpha_2 \rangle$ is a monomorphism and $c$ is Bourn-normal
to the equivalence relation $(A,\alpha_1,\alpha_2)$. For a morphism $u:S\to B$
it follows from Lemma 
\ref{lemma:standard_properties_of_commutes_relation}
that the conditions: 
\begin{enumerate}[(i)] \item $u$ and $c$ commute; \item
$\langle \alpha_1,\alpha_2 \rangle \kappa =\langle 0,1\rangle c = (c\times c)
\langle 0,1\rangle$ and $\langle \alpha_1,\alpha_2\rangle \beta u=\langle 1,1
\rangle u = (u\times u) \langle 1,1\rangle$ commute; \item $\kappa$ and $\beta
u$ commute 
\end{enumerate}
 are all equivalent.  Therefore, if the upper split
extension in \eqref{diagram:1} is faithful it follows that $u$ and $c$ commute
if and only if $u=0$ (as follows immediately from e.g.{} Proposition 2.5 of
\cite{CIGOLI_MANTOVANI:2012}).  
\end{proof}
Using these facts we obtain the following proposition which should be compared to
Theorem 7.1 of \cite{BORCEUX_BOURN:2007}:
\begin{proposition}
\label{proposition:X_centerless_c_X_centralizer_free}
Let $\C$ be a pointed protomodular category and let $X$ be an object
with trivial center, such that the generic split extension with kernel $X$ exists.
\begin{enumerate}[(i)]
\item The morphism $\conj_X$ is a
Bourn-normal monomorphism
with trivial centralizer,
and the object $[X]$ has trivial center.
\item Every $X$-groupoid is an equivalence relation.
\item The morphism $\conj_X$ is terminal in the category
of Bourn-normal monomorphisms with kernel $X$.
\item When $\C$ is in addition semi-abelian, for each internal crossed module  
$(B,X,\zeta,f)$ the morphism $f$ is a normal monomorphism.
\end{enumerate}
\end{proposition}
\begin{proof}
The claims (i) and (ii) are direct corollaries of the previous lemmas, while
the (iv) is obtained from (ii) via the equivalence of categories
between internal groupoids and internal crossed modules. The final claim follows by noting that
there is an equivalence of categories between $X$-groupoids which are
equivalence relations and Bourn-normal monomorphisms with domain $X$, and under
this equivalence the terminal $X$-groupoid is sent to $\conj_X$.  
\end{proof}
Recalling that for a group homomorphism $f:X\to Y$ the group 
$[X,Y,f]$ is the subgroup of $\Aut(X)\times \Aut(Y)$ consisting of those
pairs of automorphisms $(\theta,\phi)$ such that $f\theta =\phi f$, and $q_1$ and
$q_2$ are the first and second projections, one
sees that the following proposition applied to the category
of groups explains why for a normal subgroup $S$ of a group $X$ if the
centralizer of $S$ in $X$ is trivial, then each automorphism of $S$ has at most
one extension to $X$.  
\begin{proposition}
\label{proposition:m_trivial_centralizer=>q_1_mono}
Let $\C$ be a pointed protomodular category, and let 
$m:S\to X$ be a normal monomorphism in $\C$ such that the generic split
extensions with kernel $S$ and $(S,X,m)$ exist in $\C$ and $\C^\two$,
respectively.
If $m$ has trivial centralizer, then $q_1 : [S,X,m] \to [S]$ is a monomorphism.
\end{proposition}
\begin{proof} By considering the diagram \eqref{diagram:bourn_normal}, via
 the universal property of the split extension classifiers of $(S,X,m)$ and
 $S$,
there is a (unique) morphism $u : X\to [S,X,m]$ making the diagram 
\[
\xymatrix{
S\ar[d]_{\conj_S}\ar[r]^{m}&X\ar@{=}[r]\ar[d]^{u}&X\ar[d]^{\conj_X}\\
[S]&[S,X,m]\ar[r]_-{q_2}\ar[l]^-{q_1}&[X]
}
 \]
 commute.  Let $\kappa : K \to [S,X,m]$ be
the kernel of $q_1$. We will show that $K=0$. Consider the diagram
\[ \xymatrix{
I \ar[r] \ar[d] & J \ar[r]\ar[d]^{\lambda} & K \ar[r]\ar[d]^{\kappa} & 0
\ar[d]\\ S \ar[r]_{m} & X \ar[r]_-{u} & [S,X,m] \ar[r]_{q_1} & [S] }
\]
in which
all squares are pullbacks. Since $S$ has trivial center and $\conj_S=q_1um$ it
follows that $I=\Ker(\conj_S)=Z(S)=0$. This means that $m$ and
 $\lambda$ commute (see e.g. Proposition 3.3.2 of
 \cite{BORCEUX_BOURN:2004}) and hence $J=0$. However, since $X$ has trivial center it
follows, by Proposition
\ref{proposition:X_centerless_c_X_centralizer_free},
that $\conj_X$ is a normal monomorphism, and hence
since $q_2$ is a monomorphism (see Proposition 4.5 of \cite{GRAY:2013b})
it follows that $u$ is a normal monomorphism too. This means
that $u$ and
$\kappa$ commute and therefore so do $\conj_X$ and $q_2\kappa$. But, according to
Proposition 
\ref{proposition:X_centerless_c_X_centralizer_free},
 $\conj_X$ has trivial centralizer, and hence $K=0$ as desired.
\end{proof}
The following proposition, applied to the category of groups, implies that
each automorphism $\theta$ of a group
$X$ with trivial center admits a unique extension $\varphi$  to the
automorphism group $\Aut(X)$ in such a way that $c_X\theta = \varphi c_X$.
\begin{proposition}
\label{proposition:X_centerless_[c_X]=c_[X]}
Let
$\C$ be a pointed protomodular category and let $X$ be an object in $\C$
such that the generic split extension with kernel $X$ exists.
If the generic split extension with kernel $(X,[X],\conj_X)$ exists in $\C^\two$,
and $X$ has trivial center, then there is a unique morphism $\varphi : [X]\ltimes X \to
[[X]]\ltimes [X]$ making the diagram 
\[ \xymatrix{ X\ar[d]_{\conj_X} \ar[r]^-{k} &
[X]\ltimes X\ar[d]^-{\varphi} \ar@<0.5ex>[r]^-{p_1} &
[X]\ar@<0.5ex>[l]^-{i}\ar[d]^-{\conj_{[X]}}\\ [X] \ar[r]^-{k} & [[X]]\ltimes [X]
\ar@<0.5ex>[r]^-{p_1} & [[X]]\ar@<0.5ex>[l]^-{i} } \] a generic split extension
in $\C^\two$.  
\end{proposition}
 
\begin{proof} Let $X$ be an object with trivial
center. Consider the diagram 
\[ \xymatrix{ X \ar[r]^-{k}\ar[d]_{\conj_X} & [X]\ltimes X
\ar@<0.5ex>[r]^-{p_1}\ar[d]^{\langle p_1,p_2\rangle} &
[X]\ar@<0.5ex>[l]^-{i}\ar@{=}[d] \ar@/^4ex/[rr]^{1_{[X]}}\ar[r]^-{\theta} &
[X,[X],\conj_X]\ar[d]^-{q_2} \ar[r]^-{q_1} & [X]\\ [X] \ar[r]^-{\langle 0,1\rangle}
& [X]\times [X] \ar@<0.5ex>[r]^-{\pi_1} & [X] \ar@<0.5ex>[l]^-{\langle
1,1\rangle} \ar[r]_{\conj_{[X]}} & [[X]] }
 \]
 where the morphism in $\C^{\two}$ displayed on the right is the unique
 morphism obtained from the split extension on the left via the universal
 property of the split extension
 classifier of $(X,[X],c_X)$ and which has lower morphism $\conj_{[X]}$
according to Lemma 4.2 in \cite{GRAY:2013b}.  The universal property of the
split extension classifier $[X]$ then shows that $q_1\theta = 1_{[X]}$. However,
since by Proposition
\ref{proposition:X_centerless_c_X_centralizer_free}
the
morphism $\conj_{X}$ is a Bourn-normal monomorphism with trivial centralizer, it follows,
by the previous proposition,
that $q_1: [X,[X],\conj_X]\to [X]$ is a monomorphism and hence an isomorphism. This
means that $\theta$ is an isomorphism which completes the claim.  
\end{proof}
\section{Complete objects}
\label{section:complete}
In this section we study four notions of
completeness and explain how Baer's result can be recovered categorically.
For a pointed category $\C$ we call a morphism a protosplit monomorphism
\cite{BORCEUX_JANELIDZE_KELLY:2005b}
if it is the kernel of a split
epimorphism. As mentioned above we define:
\begin{definition}
\label{definition:completeness}
Let $\C$ be a pointed
category. An object $X$ is called 
\begin{enumerate}[(i)]
\item \emph{\pinj} if every protosplit monomorphism with domain $X$ is a
split monomorphism;
\item \emph{\inj} if every normal monomorphism with domain $X$ is a
split monomorphism;
\item \emph{\binj} if every Bourn-normal monomorphism with domain $X$ is a
split monomorphism;
\item \emph{\scmp} if every protosplit monomorphism with domain $X$ is a
split monomorphism with a unique section.
\end{enumerate}
 
\end{definition}
\begin{remark}
Since in a pointed Barr-exact category every Bourn-normal monomorphism is a
normal monomorphism it follows that \injness{} and \binjness{} coincide in
the pointed Barr-exact context.
\end{remark}

Most of the content of the following lemma follows from Corollary 3.3.3 of
\cite{BORCEUX_BOURN:2004}.
\begin{lemma}
\label{lemma:split_extensions_are_products}
Let $X$ be a \pinj{}
object in a pointed protomodular category $\C$ with finite limits. For each
split extension 
\[
\xymatrix{ X \ar[r]^{\kappa} & A \ar@<0.5ex>[r]^{\alpha} & B
\ar@<0.5ex>[l]^{\beta}
}
\]
there exists a morphism $\theta: B\to X$ and an
isomorphism $\psi: A \to B\times X$ making the diagram 
\[ \xymatrix{ X
\ar[r]^{\kappa} & A \ar@<0.5ex>[r]^{\alpha}\ar[d]^{\psi} & B
\ar@<0.5ex>[l]^{\beta}\\ X \ar[r]_-{\langle 0,1\rangle}\ar@{=}[u] & B\times X
\ar@<0.5ex>[r]^-{\pi_1} & B \ar@<0.5ex>[l]^-{\langle 1,\theta \rangle}
\ar@{=}[u]\\ } \] an isomorphism of split extensions.  
\end{lemma}
 
\begin{proof}
Suppose that $X$ is a \pinj{} object in a pointed protomodular category $\C$ and let 
\[
\xymatrix{ X \ar[r]^{\kappa} & A \ar@<0.5ex>[r]^{\alpha} & B
\ar@<0.5ex>[l]^{\beta} } \] be a split extension. By assumption there exists a
morphism $\lambda : A \to X$ such that $\lambda \kappa =1_X$. It easily follows
that the diagram 
\[ \xymatrix{ X \ar[r]^{\kappa} & A
\ar@<0.5ex>[r]^{\alpha}\ar[d]^{\langle \alpha,\lambda\rangle} & B
\ar@<0.5ex>[l]^{\beta}\\ X \ar[r]_-{\langle 0,1\rangle}\ar@{=}[u] & B\times X
\ar@<0.5ex>[r]^-{\pi_1} & B \ar@<0.5ex>[l]^-{\langle 1,\lambda \beta \rangle}
\ar@{=}[u]\\ } \] is a morphism of split extensions and hence by protomodularity
$\langle \alpha,\lambda\rangle$ is an isomorphism.  
\end{proof}
We will also need the following easy lemma (see e.g. Corollary 3.3.3 of
\cite{BORCEUX_BOURN:2004}).
\begin{lemma}
\label{lemma:split_bourn-normal=product_inclusion}
Let $\C$ be a pointed protomodular category and let $m: S\to X$ be a
Bourn normal monomorphism. If $m$ is a split monomorphism, then $m$ is
 a \emph{product inclusion}, that is, there exists an object $T$ and an isomorphism
 $\varphi : S\times T \to X$ such that $\varphi \langle 1, 0 \rangle = m$.
\end{lemma}
%\begin{proof}
% Let $p:X\to S$ be a splitting of $m$ and let $k : T\to X$ be the kernel of $p$.
% An easy calculation shows that $m$ and $k$ have trivial intersection and
% hence commute (see e.g. Proposition 3.3.2 of \cite{BORCEUX_BOURN:2004}).
% Writing $\varphi : S\times T\to X$ for the cooperator of $k$ and $m$, using
%the fact that $\langle 1,0\rangle$ and $\langle0,1 \rangle$ are jointly
% epimorphic, we see that the diagram
% \[
%  \xymatrix{
%   T \ar[r]^-{\langle 0,1\rangle}\ar@{=}[d] & S\times T\ar[d]^{\varphi} \ar@<0.5ex>[r]^-{\pi_1} & 
%   S\ar@<0.5ex>[l]^-{\langle 1,0\rangle}\ar@{=}[d]\\
%   T \ar[r]_{k} & X \ar@<0.5ex>[r]^{p} & S\ar@<0.5ex>[l]^{m}
%  }
% \]
% is a morphism of split extensions. The claim now follows by protomodularity.
%\end{proof}
\begin{proposition}
\label{proposition:proto-complete+centerless=strong-complete}
Let $\C$ be a pointed protomodular category. An object $X$ in $\C$ is \scmp{}
if and only if
$X$ is \pinj{} and has trivial center.
\end{proposition}
\begin{proof}
Let $X$ be a \pinj{} object in $\C$.
Since every
protosplit normal monomorphism with domain $X$ is, by Lemma
\ref{lemma:split_extensions_are_products}
up to isomorphism, of the form $\langle 1,0\rangle : X\to X\times Y$
for some $Y$, the claim follows from Lemma
\ref{lemma:char_of_centerless}
\end{proof}
\begin{remark}
As mentioned above categories in which every object is \scmp{} are called
coarsely action representable. These include the opposite category of pointed sets.
\end{remark}
\begin{proposition}
\label{proposition:implications_between_completeness}
In a pointed protomodular category the implications hold:\\
\scmpness{} $\Rightarrow$ \binjness{} $\Rightarrow$ \injness{} $\Rightarrow$ \pinjness{}.
\end{proposition}
 
\begin{proof} Trivially
\binjness{} $\Rightarrow$ \injness{} $\Rightarrow $\pinjness{}. 
Suppose $X$ is \scmp{}, we need to show that it
is \binj{}. Let $n: X\to B$ be a Bourn-normal monomorphism and let $r_1,r_2: R\to
B$ be the (projections of the) relation of which it is the zero class. This
means that there are morphisms $k:X\to R$ and $\tilde n: X\times X\to R$ such
that the diagram 
\begin{equation}
\label{diagram:equiv}
\vcenter{
\xymatrix@=2ex@!{ & X\ar[rr]^{\langle
0,1\rangle}\ar@{=}[dl]\ar@/^3ex/[ddl]^(0.3){n} && X\times X \ar[dl]_{\tilde
{n}} \ar@<0.5ex>[rr]^{\pi_1}\ar@/^3ex/[ddl]^(0.2){n\times n} && X
\ar[dl]_{n}\ar@<0.5ex>[ll]^{\langle 1,1\rangle}\ar@/^3ex/[ddl]^{n}\\ X
\ar[rr]^(0.7){k}\ar[d]_{n} && R \ar[d]_{\langle r_1,r_2\rangle}
\ar@<0.5ex>[rr]^(0.7){r_1} && B\ar@{=}[d] \ar@<0.5ex>[ll]^(0.3){s}\\ B
\ar[rr]^{\langle 0,1\rangle} && B,\times B \ar@<0.5ex>[rr]^{\pi_1} && B
\ar@<0.5ex>[ll]^{\langle 1,1\rangle} } 
}
\end{equation}
in which $s$ is the unique morphism
 $r_1s=r_2s=1_B$, is a morphism of split extensions. According to Lemma
\ref{lemma:split_extensions_are_products}
the middle split extension is isomorphic to a split extension of the form
\[ \xymatrix{ X \ar[r]^-{\langle 0,1\rangle} & B\times X \ar@<0.5ex>[r]^-{\pi_1}
& B \ar@<0.5ex>[l]^-{\langle 1,\theta \rangle} } \] and hence composing
the upper morphism of diagram \eqref{diagram:equiv} with this isomorphism we
obtain a morphism of split extensions of the form 
\[ \xymatrix{ X
\ar[r]^-{\langle 0,1\rangle} \ar@{=}[d]& X\times X
\ar@<0.5ex>[r]^-{\pi_1}\ar[d]_{\langle n\pi_1, \phi\rangle} & X
\ar@<0.5ex>[l]^-{\langle 1,1 \rangle} \ar[d]^{n}\\ X \ar[r]^-{\langle
0,1\rangle} & B\times X \ar@<0.5ex>[r]^-{\pi_1} & B, \ar@<0.5ex>[l]^-{\langle
1,\theta \rangle} } \] for some morphism $\phi: X\times X \to X$.
This implies that $\phi$ is a splitting of $\langle 0,1\rangle$ and
hence must be $\pi_2$ (since $X$ is \scmp{}). It now follows that
$\theta n =\pi_2 \langle 1,\theta\rangle n = 
\pi_2 \langle n\pi_1,\pi_2\rangle \langle 1,1\rangle = 1_X$ as desired.
\end{proof}
 Since in an abelian category every monomorphism is a normal
monomorphism and every protosplit monomorphism is a split monomorphism we see
that: 
\begin{proposition}
\label{proposition:abelian_seperates_completness}
Let $\C$ be an abelian category.  
\begin{enumerate}[(i)]
\item Every object $X$ satisfies the condition: if 
$\kappa : X \to A$ is the kernel of a split epimorphism
which is also a normal epimorphism, then $\kappa$ is a
split monomorphism;
\item Every object is \pinj;
\item An object is \inj{} if and only if it is (regular) injective;
\item An object
is \scmp{} if and only if it is a zero object.  
\end{enumerate}
More generally (i), (iii) and (iv) hold if $\C^{\text{op}}$ is semi-abelian. 
\end{proposition}
\begin{proof}
Noting that (i) and (ii) are equivalent for an abelian category.
We prove the dual of (i), (iii) and (iv) for $\C$ a semi-abelian category.
To prove the dual of (i) suppose that $\alpha : A\to B$ is the cokernel of a
normal split monomorphism $\kappa : X\to A$.
If $\lambda$ is a splitting of $\kappa$, then the diagram
\[ \xymatrix{ X \ar[r]^{\kappa} & A
\ar[r]^{\alpha}\ar[d]^{\langle \alpha,\lambda\rangle} & B\\
X \ar[r]_-{\langle 0,1\rangle}\ar@{=}[u] & B\times X
\ar[r]_-{\pi_1} & B \ar@{=}[u]\\ } \]
is a morphism of short exact sequences and hence by the short
five lemma $\langle \alpha,\lambda \rangle$ is an isomorphism and
$\alpha$ a split epimorphism. The dual of (iii) is a standard
fact about regular projective objects in regular categories. The dual of (iv)
follows easily from the fact
that if $X$ is
an object in $\C$, then $\langle 1,1\rangle$ and $\langle 1,0\rangle$
are both sections of the normal epimorphism $\pi_1$, and
$\langle 1,1\rangle=\langle 1,0\rangle$
if and only if $1_X=0$ if and only if $X$ is a zero object. 
\end{proof}
 On the other hand in the category $\Rng$ of not necessarily
unitary rings, \pinjness{}, \injness{} and \scmpness{} coincide and are
equivalent to being unitary. Recall that as mentioned in the introduction
it is known that \injness{} is equivalent to being unitary.
\begin{proposition} Let $X$ be an object in
$\Rng$. The object $X$ is a unitary if and only if it satisfies any of
Definition 
\ref{definition:completeness}
(i)-(iv).  
\end{proposition}
\begin{proof} Let $X$ be a ring. If $X$ is \pinj{}, then we see that the
morphism $k$, which forms part of the split extension 
\[ \xymatrix{ X
\ar[r]^-{k} & \mathbb{Z} \ltimes X \ar@<0.5ex>[r]^-{p_1} & \mathbb{Z}
\ar@<0.5ex>[l]^-{i} } \] obtained by ``adding a multiplicative identity to
$X$'', splits. Since a surjective ring homomorphism sends a multiplicative
identity to a multiplicative identity it follows that $X$ is unitary. Conversely
if $X$ is unitary and $X$ is an ideal of $Y$, then the map $l: Y\to X$ defined
$y\mapsto ey$ (where $e$ is the multiplicative identity in $X$) is a morphism
which splits the inclusion. Indeed, using that $e$ is a multiplicative identity
it immediately follows that it splits the inclusion, and since multiplication
by an element is always an abelian group homomorphism one only needs to check
that $l$ preserves multiplication. However since for any $y \in Y$, $ey$ is in
$X$ and hence $eye=ey$ it follows that $l(y_1y_2)=
ey_1y_2=ey_1ey_2=l(y_1)l(y_2)$. Finally note that $X$ has trivial center since
$ex=x$ for all $x \in X$.  
\end{proof}
 
\begin{remark} The previous proof can be
easily adapted to prove that \pinjness{}, \injness{} and \scmpness{}
coincide and
are equivalent to being unitary for the category of algebras over an
arbitrary ring.
It is also easy to show that a morphism $f: X\to Y$
in $\Rng$ is a unitary ring morphism if and only if it is \scmp{} in the
category of morphisms of $\Rng$. In this way one can recover the category of
unitary rings from the category of rings.\end{remark}

\begin{remark} The Propositions
\ref{proposition:implications_between_completeness}
and
\ref{proposition:abelian_seperates_completness}
show that in
general for $\mathcal{SC}$, $\mathcal{C}$ and $\mathcal{PC}$ the classes of
\scmp{}, \inj{} and \pinj{} objects in a semi-abelian category the inclusions
$\mathcal{SC} \subset \mathcal{C} \subset \mathcal{PC}$ are strict.  However we
will recall the above mentioned fact that a group is \scmp{} if and only if it
is what we called \inj{}, and we will show that a Lie algebra over a commutative
ring is \scmp{} if and only if it is \inj{} if and only if it is \pinj{}. We
will also show that there are  groups, in particular the group $\z{2}$ is such
an example, that are \pinj{} but not \inj{}.  
\end{remark}
%\begin{proposition}
%Let $\C$ be a semi-abelian category. An object $(A,A',\alpha)$ in the category
%$\text{Reg}(\C)$ is:
%\begin{enumerate}[(i)]
%\item \pinj{} if $\alpha=1_A$ and $A$ is \pinj{} in $\C$;
%\item \inj{} if $\alpha=1_A$ and $A$ is \inj{} in $\C$;
%\item \binj{} if $A$ is \inj{}  and $A'$ is \scmp{} in $\C$;
%\item \scmp{} if $A$ and $A'$ are \scmp{} in $\C$;
%\end{enumerate}
%\end{proposition}
%\begin{proof}
%
%\end{proof}
         %\begin{definition} A morphism $m:S\to X$ is called \emph{\awe{}} if it
         %is a split monomorphism and for each normal monomorphism $n: S\to Y$
         %there exists an object $A$ and morphisms $u:X\to A$ and $v:Y\to A$
         %such that the diagram \[ \xymatrix{ S \ar[r]^{n}\ar[d]_{m} &
         %Y\ar[d]^{v}\\ X \ar[r]_{u} & A } \] is a pullback and $u$ is a normal
         %monomorphism.  
         %\end{definition}
\begin{proposition}
\label{proposition:product_factor_of_complete}
Let $\C$ be a pointed protomodular category, let $S$ and $T$ be objects in $\C$, and let
$X=S\times T$. If $X$ is \pinj{}, \inj{}, \binj{} or \scmp{}, then so are both $S$ and
$T$.  
\end{proposition}
 
\begin{proof} For the cases where $X$ is \pinj{},
\inj{} or \binj{} note that if $n:S\to Y$ is a protosplit/normal/Bourn-normal
monomorphism, then $n\times 1_T$ is too. This means that in each case $n\times
1_T$ is a split monomorphism. Since the diagram 
\[ \xymatrix{ S\ar[d]_{\langle
1,0\rangle} \ar[r]^{n} & Y \ar[d]^{\langle 1,0\rangle}\\ S\times T
\ar[r]_{n\times 1_T} & Y\times T } \] commutes and the composite of two split
monomorphisms is a split monomorphism it follows that $\langle 1,0\rangle n =
(n\times 1_T) \langle 1,0\rangle$ is a split monomorphism and hence $n$ is
too. The claim now follows from Corollary 
\ref{corollary:factors_of_a_product_centerless_<=>_product_centerless}
and Proposition
\ref{proposition:proto-complete+centerless=strong-complete}
.  
\end{proof}
The converse of the implications of the previous proposition
don't hold in general, but do hold in certain categories.  
\begin{example} The product of \scmp{}
objects need not be \pinj{}. If $X$ is a non-zero \scmp{} group, then 
$X\times X$ has trivial center and hence is \scmp{} if and only if it is
\pinj{}. However the
automorphism $X\times X \to X\times X$ defined by $(x,y) \mapsto (y,x)$ is
 clearly not inner (which is sufficient via Baer's theorem). On the other hand since the product of unitary rings is
unitary we see that product of \scmp{} objects in the category of rings is
always \scmp{}. It is also easy to check that the converse of the
implications of the previous proposition do hold in each abelian category.  
\end{example}
Recall that every unital variety of universal algebras has centers of objects.
Recall also, that in a unital category $\C$, central subobjects are normal
whenever $\C$
satisfies the condition that for each composite $f=\alpha n$ where $\alpha$
is
a split epimorphism and $n$ is a normal monomorphism $f$ is a normal
monomorphism as soon as it is a monomorphism. In particular every semi-abelian
variety of universal algebras and every semi-abelian algebraically cartesian
closed category satisfies these properties. Furthermore, we have (the probably
known fact which we couldn't find a reference for): 
\begin{proposition} Let $\C$
be a pointed protomodular category. If $m:S\to X$ is a central monomorphism in
$\C$, then $m$ is Bourn-normal.  
\end{proposition}
 
\begin{proof} Suppose $m:S\to
X$ is a central monomorphism with cooperator $\varphi: S\times X\to X$. It
follows that the diagram 
\[ \xymatrix{ S\ar[d]_{m} \ar[r]^{\langle 1,0\rangle} &
S\times X \ar@<0.5ex>[r]^{\pi_2}\ar[d]^{\langle\varphi,\pi_2\rangle} &
X\ar@<0.5ex>[l]^{\langle 0,1\rangle} \ar@{=}[d]\\ X \ar[r]_{\langle 1,0\rangle}
& X\times X \ar@<0.5ex>[r]^{\pi_2} & X\ar@<0.5ex>[l]^{\langle 1,1\rangle} } \]
is a morphism of split extensions and hence by protomodularity $\langle
\varphi,\pi_2\rangle$ is a monomorphism and $m$ a Bourn-normal monomorphism.
\end{proof}
For an object $X$ in a pointed weakly unital category consider the condition:
\begin{condition} \label{condition:nice_center}\ 
The center $z_X:Z(X)\to X$ exists, $\langle z_X,z_X\rangle : Z(X)\to X\times X$ is a normal monomorphism,
 and the cokernels of $z_X$ and $\langle z_X,z_X\rangle$ exist.
\end{condition}
Note that the above condition holds for each object in a semi-abelian category
admitting centers and hence in each semi-abelian variety
of universal algebras. On the other hand it holds for each abelian object in
a unital category.
\begin{theorem}
\label{theorem:proto-complete_product_decomposition}
Let $\C$ be a pointed protomodular category and let $X$ be an object in $\C$.
If
$X$ is \pinj{} (respectively \inj{} or \binj{}) and satisfies Condition
\ref{condition:nice_center},
then $X$ is the product of its
center $Z(X)$ which is an abelian \pinj{} (respectively \inj{} or \binj{}) object
and the quotient object $X/Z(X)$ which is a \scmp{} object.  
\end{theorem}
\begin{proof} Let $X$ be a \pinj{} object in $\C$ and let $(Z,z)$ be the center
of $X$. 
Consider the diagram 
\[ \xymatrix{ &
Z\ar@{=}[r]\ar[d]_{\langle z,z\rangle} & Z \ar[d]^{z}\\ X \ar[r]^-{\langle
0,1\rangle}\ar@{=}[d] & X\times X \ar@<0.5ex>[r]^-{\pi_1}\ar[d]^{\varphi} & X
\ar@<0.5ex>[l]^-{\langle 1, 1\rangle}\ar[d]^{c}\\ X \ar[r]_-{\langle 0
,1\rangle} & B\times X \ar@<0.5ex>[r]^-{\pi_1} & B\ar@<0.5ex>[l]^-{\langle
1,\theta\rangle}
}
\]
in which $c$ and $\varphi$
are cokernels of $z$ and $\langle z,z\rangle$, respectively, and the lower
induced split extension is of the form presented by Lemma
\ref{lemma:split_extensions_are_products}
. This means that the diagram 
\[
\xymatrix{
X \ar[r]^-{\langle 1,0\rangle}\ar[dr]& X\times X
\ar[d]^{\pi_2\varphi} & X \ar[l]_-{\langle 0,1\rangle}\ar[dl]^{1_X}\\ & X & 
} 
\]
commutes and so the morphism $\pi_2\varphi \langle 1,0\rangle$ is central.
It follows there exists $\bar f: X\to Z$ such that
$z\bar f =\pi_2\varphi\langle 1,0\rangle$ and hence
$c\pi_2\varphi \langle 1,0\rangle = 0$. Therefore, since
the morphisms $\langle 1,0\rangle $ and $\langle 0,1\rangle$ are jointly
epimorphic it follows that $c\pi_2\varphi = c\pi_2$ and hence that $c\theta c =
c \pi_2 \langle 1,\theta\rangle c = c\pi_2\varphi \langle 1,1\rangle = c\pi_2
\langle 1,1\rangle =c$. This means that $c$ is a split epimorphism which means
that its kernel being central is a \emph{product inclusion}. Indeed, if $\gamma
: B\times Z\to X$ is the cooperator of $z$ and $\theta$, then since $\langle
1,0\rangle$ and $\langle 0,1\rangle$ are jointly epimorphic it follows that the
diagram
\[
\xymatrix{
Z \ar@{=}[d]\ar[r]^-{\langle 0,1\rangle} & B\times Z
\ar@<0.5ex>[r]^-{\pi_1}\ar[d]^{\gamma} & B\ar@<0.5ex>[l]^-{\langle
1,0\rangle}\ar@{=}[d]\\ Z \ar[r]_{z} & X \ar@<0.5ex>[r]^{c} &
B\ar@<0.5ex>[l]^{\theta}
}
\]
is a morphism of split extensions and $\gamma$ is
an isomorphism. 
Since $\langle 0,1\rangle:Z \to B\times Z$ is the center of $B\times Z$ it
follows by Proposition
\ref{proposition:center_of_products}
 that $B$ has trivial center. The claim now follows from Propositions 
\ref{proposition:proto-complete+centerless=strong-complete},
\ref{proposition:implications_between_completeness}
and
\ref{proposition:product_factor_of_complete}
.
\end{proof}
 The previous
theorem raises the question of whether the product of a \pinj{} abelian object
and a \scmp{} object is necessarily \pinj. This turns out to be false in
general. Let $S_3$ be the symmetric group on a three element set. It is
well-known (and easy to check) that $S_3$ is \inj{}=\scmp{}.  If the group
$X=\z{2}\times S_3$ is \pinj{}, then by the previous theorem since $\langle
1,0\rangle: \z{2} \to X$ is the center of $X$ we must have the $\Aut(X) \cong
S_3\cong \Aut(S_3)$ and so every automorphism of $X$ must be of the form
$1\times \phi$ where $\phi \in \Aut(S_3)$. However, recalling that the group
$S_3$ is isomorphic to the semi-direct product $\z{3} \ltimes_\theta \z{2}$
where $\theta: \z{2} \to \Aut(\z{3})$ is (the isomorphism) defined by
$1 \mapsto (x \mapsto 2\cdot
x)$ and hence forms part of a split extension 
\[ \xymatrix{ \z{3} \ar[r]^{k} &
S_3 \ar@<0.5ex>[r]^{p} & \z{2}\ar@<0.5ex>[l]^{i} } \] and writing $a:\z{2}\times
\z{2} \to \z{2}$ for the addition morphism of $\z{2}$, we see that the morphism
$\langle a(1_{\z{2}} \times p),\pi_2\rangle : X\to X$ is an automorphism which
is not of this form. On the other hand we do have:
\begin{proposition}
\label{proposition:product_ab_binj_scmp}
Let $\C$ be a pointed protomodular category in which centralizers of
Bourn-normal
monomorphisms exists and are Bourn-normal. Suppose $A$ is an abelian and
\binj{}
object of $\C$, and $B$ is a \scmp{} object of $\C$.
If $\hom(B,A)=\{0\}$,
 then $B\times A$ is \binj{}.
\end{proposition} 
\begin{proof}
Suppose $n:B\times A\to Y$ is a Bourn-normal monomorphism. Since
the pullback of the centralizer of a monomorphism $f$ along $f$ is the
center of its domain it follows that
there is a unique morphism $m$ making the diagram
\[
\xymatrix{
A\ar[r]^-{m}\ar[d]_{\langle 0,1\rangle} & Z_Y(B\times A,n)\ar[d]^{z_n}\\
B\times A\ar[r]_{n} & Y
}
\]
a pullback. This means that $n\langle 0,1\rangle$
(being the binary meet of Bourn-normal monomorphisms) is Bourn-normal and
hence by
Lemma \ref{lemma:split_bourn-normal=product_inclusion}
there is an isomorphism $\theta : Y\to C\times A$ such that
$\theta n \langle 0,1\rangle = \langle 0,1\rangle$.
Therefore, since $\hom(B,A)=\{0\}$ it follows that 
$\pi_2 \theta n \langle 1,0\rangle = 0$ and hence
that $\theta n \langle 1,0\rangle = \langle \tilde n,0\rangle$
where $\tilde n = \pi_1 \theta n \langle 1,0\rangle$.
Since $\langle 1,0\rangle$ and $\langle 0,1\rangle$ are
jointly epimorphic, it follows that
$\theta n = \tilde n \times 1$. Noting that the diagram
 \[
\xymatrix{
B \ar[r]^{\tilde n} \ar[d]_{\langle 1,0 \rangle} &
 C\ar[d]^{\langle 1,0\rangle}\\
B\times A \ar[r]_{\tilde n \times 1} & C\times A
}
\]
is a pullback, we see that $\tilde n$ is a Bourn-normal monomorphism and hence
it and $n$ are split monomorphisms.

\end{proof}
\begin{lemma}
\label{lemma:conjugation_splits}
Let $\C$ be a pointed protomodular category and let $X$ be an object
admitting a
generic split extension with kernel $X$.  The following are equivalent:
\begin{enumerate}[(a)]
\item the
generic split extension with kernel $X$ is of the form 
\[ \xymatrix{ X
\ar[r]^-{\langle 0,1\rangle} & [X]\times X \ar@<0.5ex>[r]^-{\pi_1} &
[X]\ar@<0.5ex>[l]^-{\langle 1,\theta\rangle} } \] for some morphism $\theta$;
\item $\conj_X$ is a split epimorphism.
\end{enumerate}
\end{lemma}
\begin{proof}
(a) $\Rightarrow$ (b):
Consider the diagram 
\[ \xymatrix{ X\ar@{=}[d] \ar[r]^-{\langle 0,1\rangle} &
[X]\times X \ar[d]^{\theta \times 1} \ar@<0.5ex>[r]^-{\pi_1} &
[X]\ar[d]^{\theta}\ar@<0.5ex>[l]^-{\langle 1,\theta\rangle}\\ X\ar@{=}[d]
\ar[r]^-{\langle 0,1\rangle} & X\times X\ar[d]^{\varphi} \ar@<0.5ex>[r]^-{\pi_1}
& X\ar@<0.5ex>[l]^-{\langle 1,1\rangle}\ar[d]^{\conj_X}\\ X \ar[r]^-{\langle
0,1\rangle} & [X]\times X \ar@<0.5ex>[r]^-{\pi_1} & [X].\ar@<0.5ex>[l]^-{\langle
1,\theta\rangle} } \] where the lower part is the unique morphism into the
generic split extension with kernel $X$ which defines $\conj_X$. The universal
property of the generic split extension implies that the composite $\conj_X\theta =
1_{[X]}$.\\[2ex] (b) $\Rightarrow$ (a): Suppose that $\conj_X$ is a split
epimorphism with splitting $\theta: [X]\to X$ and consider the diagram 
\[
\xymatrix{ X\ar@{=}[d] \ar[r]^-{\langle 0,1\rangle} & [X]\times X \ar[d]^{\theta
\times 1} \ar@<0.5ex>[r]^-{\pi_1} & [X]\ar[d]^{\theta}\ar@<0.5ex>[l]^-{\langle
1,\theta\rangle}\\ X\ar@{=}[d] \ar[r]^-{\langle 0,1\rangle} & X\times
X\ar[d]^{\varphi} \ar@<0.5ex>[r]^-{\pi_1} & X\ar@<0.5ex>[l]^-{\langle
1,1\rangle}\ar[d]^{\conj_X}\\ X \ar[r]^-{k} & [X]\ltimes X \ar@<0.5ex>[r]^-{p_1} &
[X]\ar@<0.5ex>[l]^-{i} } \] in which the lower part is the unique morphism into
the generic split extension with kernel $X$ which defines $\conj_X$.
 The claim now
follows from the fact that since $\conj_X \theta = 1_{[X]}$, the split
short five lemma implies that
$\varphi(\theta \times 1)$
 is an isomorphism.
\end{proof}
\begin{theorem}
\label{theorem:char_of_proto-complete}
Let
$\C$ be a pointed protomodular category.
For an object $X$ in $\C$ the
following are equivalent: 
\begin{enumerate}[(a)]
\item $X$ is \pinj{} and satisfies Condition \ref{condition:nice_center},
\item the
generic split extension with kernel $X$ exists and is of the form 
\[ \xymatrix{ X
\ar[r]^-{\langle 0,1\rangle} & [X]\times X \ar@<0.5ex>[r]^-{\pi_1} &
[X]\ar@<0.5ex>[l]^-{\langle 1,\theta\rangle} } \] for some morphism $\theta$.
\item the
generic split extension with kernel $X$ exists and $\conj_X$ is a split
epimorphism.
\end{enumerate}
\end{theorem}
\begin{proof}
The equivalence of (b) and (c) follow from the previous lemma.
Suppose the assumptions in (a) hold. According to Theorem
\ref{theorem:proto-complete_product_decomposition}
there exists objects
$A$ and $[X]$ in $\C$ such that $A$ is abelian and \pinj{}, $[X]$ is \scmp{},
and $X=[X]\times A$.
We will show that
\begin{equation}
\label{diagram:proto-complete_classifier}
\vcenter{
\xymatrix{
X\ar[r]^-{\langle 0,1\rangle} & [X]\times X \ar@<0.5ex>[r]^-{\pi_1} &
[X] \ar@<0.5ex>[l]^-{\langle 1,\langle 1,0\rangle \rangle} 
}
}
\end{equation}
is a generic split extension with kernel $X$. For any
split extension with kernel $X$, which by Lemma
\ref{lemma:split_extensions_are_products}
is (up to isomorphism) of the form displayed at the top of the
following diagram, the composite of the morphisms
\[
\xymatrix{
X\ar[r]^-{\langle 0,1\rangle}\ar@{=}[d] & 
B\times X \ar@<0.5ex>[r]^-{\pi_1}
\ar[d]|{\langle \pi_1,\langle \pi_1\pi_2,\pi_2\pi_2-g\pi_1\rangle\rangle} &
B \ar@<0.5ex>[l]^-{\langle 1,\langle f,g\rangle \rangle}\ar@{=}[d]\\ 
X\ar[r]^-{\langle 0,1\rangle}\ar@{=}[d] &
B\times X \ar@<0.5ex>[r]^-{\pi_1}\ar[d]^{f\times 1} &
B \ar@<0.5ex>[l]^-{\langle 1,\langle f,0\rangle \rangle}\ar[d]^{f}\\
X\ar[r]^-{\langle 0,1\rangle} & [X]\times X \ar@<0.5ex>[r]^-{\pi_1} &
[X], \ar@<0.5ex>[l]^-{\langle 1,\langle 1,0\rangle \rangle} 
}
\]
gives a morphism to \eqref{diagram:proto-complete_classifier}.
It remains only to show that this is the unique such morphism.
By protomodularity it is sufficient to show that for a morphism
of split extensions
\[
\xymatrix{
X\ar[r]^-{\langle 0,1\rangle}\ar@{=}[d] &
B\times X \ar@<0.5ex>[r]^-{\pi_1}\ar[d]^{u} &
B \ar@<0.5ex>[l]^-{\langle 1,\langle f,g\rangle \rangle}\ar[d]^{v}\\
X\ar[r]^-{\langle 0,1\rangle} & [X]\times X \ar@<0.5ex>[r]^-{\pi_1} &
[X], \ar@<0.5ex>[l]^-{\langle 1,\langle 1,0\rangle \rangle} 
}
\]
we must have $v=f$.
To do so note that the diagram
\[
\xymatrix{
[X]\ar[r]^(0.60){\langle 1,0\rangle}\ar[d]_{1_{[X]}}
\ar@/^17.22pt/[rr]^{\langle 0,\langle 1,0\rangle \rangle}&
X\ar[rd]_{\langle 0,1\rangle}\ar[r]^(0.40){\langle 0,1\rangle}&
B\times X\ar[d]^{u}\\
[X]&X\ar[l]^{\pi_1}&[X]\times X\ar[l]^{\pi_2}
}
\]
commutes, and so since $[X]$ is \scmp{}, and $\langle 0,\langle 1,0\rangle\rangle$
is normal with splitting $\pi_1\pi_2$
it follows that
$\pi_1\pi_2u=\pi_1\pi_2$. This means that 
\[v = \pi_1\pi_2 \langle 1,\langle 1,0\rangle v
    = \pi_1\pi_2 u \langle 1,\langle f,g \rangle\rangle
    = \pi_1\pi_2\langle 1,\langle f,g\rangle\rangle=f\]
as desired. The claim now follows by noting that: (b) implies that every
protosplit monomorphism with domain $X$
factors through $\langle 0,1\rangle:X\to [X]\times X$
and hence splits; in the protomodular context (c) implies that the morphism
$\conj_X$ and $\varphi$ in
\eqref{diagram:conjugation_morphism}, whose kernels are $z_X$ and
$\langle z_X,z_X\rangle$ respectively, are normal epimorphisms.
\end{proof}

\begin{corollary}
\label{corollary:X_inj_=>_X_=_Z(X)x[X],_[X]_complete}
Let $\C$ be a pointed Barr-exact protomodular category admitting
centers.
If $X$ is \inj{} (respectively \pinj{}), then,
$[X]$, the split extension classifier for $X$ exists and is \scmp{},
$Z(X)$ is an abelian \inj{} (respectively \pinj{}) object, 
and $X \cong Z(X)\times [X]$.  
\end{corollary}
 
\begin{corollary}
Let
$\C$ be a pointed Barr-exact protomodular category admitting centers.
An object $X$ in $\C$ is
\inj{} and abelian, if and only if the split extension classifier for $X$ is the zero
object.
\end{corollary}
\begin{proof} Let $X$ be an object. Since every \inj{} object has a split
extension classifier, we may assume that the split extension classifier for $X$
exists, and hence so does $\conj_X$.  Recall
that $z_X$ is the kernel of $\conj_X$, $X$ is abelian if and only if $z_X$ is an
isomorphism, and according to Theorem 
\ref{theorem:char_of_proto-complete},
$X$ is \pinj{} if and
only if $\conj_X$ splits.  Now, if $[X] \cong 0$ then $\conj_X =0$ and hence is a split
epimorphism and its kernel $z_X$ is an isomorphism. Conversely, if $z_X$ is an
isomorphism and $\conj_X$ is a split epimorphism, then $\conj_X$ being the cokernel of
$z_X$ is zero and $[X] \cong 0$.  
\end{proof}
 
%\begin{remark} The equivalent
%conditions in Theorem 
%\ref{theorem:char_of_proto-complete}
%suggest a condition on an object $X$ in an action
%representable category which is weaker than requiring it to be \pinj{}. Indeed
%we could require $\conj_X$ to be a regular/normal epimorphism rather than a split
%epimorphism.
%We will not study this condition further here.
% other than to mention
%that it admits a reformulation, which could be studied in an arbitrary
%semi-abelian category: For each split extension 
%\[ \xymatrix{ X \ar[r]^{\kappa}
%& A \ar@<0.5ex>[r]^{\alpha} & B \ar@<0.5ex>[l]^{\beta} } \] there exists a
%morphism of split extensions 
%\[ \xymatrix{ X\ar[d]_{\langle 1,1\rangle}
%\ar[r]^{\tilde \kappa} & \tilde A \ar@<0.5ex>[r]^{\tilde \alpha}\ar[d]^{f} &
%\tilde B \ar@<0.5ex>[l]^{\tilde \beta}\ar[d]^{g}\\ X\times X
%\ar[r]^-{\kappa\times \langle 0,1\rangle} & A\times (X\times X)
%\ar@<0.5ex>[r]^-{\alpha\times \pi_1} & B\times X \ar@<0.5ex>[l]^-{\beta\times
%\langle 1,1\rangle} } \] such that $\pi_2 g$ is a regular epimorphism.
%\end{remark}

Let us recall the following two facts (rephrased with our terminology): 

\begin{proposition} If $X$
is a group such that $\Aut(X)\cong 0$, then $X$ is isomorphic to $0$ or $\z{2}$.
The only (up to isomorphism) non-zero abelian \pinj{} group is $\z{2}$.
\end{proposition}
\begin{proof} Since $X$ is abelian it follows that the map
sending $x$ to $-x$ is an automorphism and hence must be equal to $1_X$. This
means that for each $x$ in $X$, $0 = x + -x = x + x = 2x$ and so $X$ is a vector
space over $\z{2}$. Since any vector space has a non-trivial automorphism
corresponding to any non-trivial permutation of elements in its basis it follows
that $X$ has dimension at most one and hence is isomorphic to $\z{2}$ or $0$.
\end{proof}
 
\begin{proposition} If $X$ is a Lie algebra over a commutative ring
$R$ such that $\Der(X)\cong 0$, then $X \cong 0$. There are no non-zero abelian
\pinj{} Lie algebras.  
\end{proposition}
\begin{proof} Since $X$ is abelian it
follows that any $R$-linear map is a derivation so that in particular the
identity
map is a derivation and hence must be equal to $0$.  
\end{proof}
\begin{proposition} Let $\C$ be an anti-additive action representable
semi-abelian category. If $[X] \cong 0$, then $X \cong 0$.  
\end{proposition}
\begin{proof} The claim is immediate since $X$ is necessarily abelian and hence
isomorphic to $0$.  
\end{proof}

Recall that for an object $X$, in a pointed protomodular category
admitting a generic split extension with kernel $X$,
the object $\Out(X)$ is codomain of the cokernel $q:[X] \to \Out(X)$ of $\conj_X$.
The following theorem should be compared to Theorems 9.1 and 9.2 of
\cite{BORCEUX_BOURN:2007}.
\begin{theorem}
\label{theorem:char_of_scmp_in_action_rep_cat}
Let $\C$ be a pointed protomodular category. For an object $X$
the following are equivalent: 
\begin{enumerate}[(a)]
\item $X$ is \scmp{};
\item the generic split extension with kernel $X$ exists and 
the conjugation morphism $\conj_X : X \to [X]$ is an isomorphism;
\item the split extension
\begin{equation}
\label{diagram:conjugation_split_extension}
\vcenter{
\xymatrix{ X \ar[r]^-{\langle 0,1\rangle} & X\times X
\ar@<0.5ex>[r]^-{\pi_1} & X\ar@<0.5ex>[l]^-{\langle 1,1\rangle} }
}
\end{equation}
is a
generic split extension; 
\item the generic split extension with kernel $X$ exists, $X$
has trivial center, and $\Out(X)$ is trivial.  \setcounter{tmp}{\arabic{enumi}}
\end{enumerate}
If in addition every abelian \inj{} object in $\C$ is a zero object,
then these conditions are further equivalent to:
\begin{enumerate}[(a)] \setcounter{enumi}{\arabic{tmp}} \item 
$X$ is \inj{} and satisfies Conditon \ref{condition:nice_center}.
\end{enumerate}
\end{theorem}
\begin{proof}
Considering the definition of $\conj_X$ and recalling that the kernel of $\conj_X$ is center
of $X$ and cokernel of $\conj_X$ is $\Out(X)$ it easily follows that (b), (c) and (d) are
equivalent. Using again the fact that the center of $X$ is kernel of $\conj_X$ the
equivalence of (a) and (b) follows from Theorem 
\ref{theorem:char_of_proto-complete}
.
 The final claim
follows from Theorem
\ref{theorem:proto-complete_product_decomposition} since it implies that
$Z(X)$ is abelian and \inj{} and hence by assumption a zero object.
\end{proof}
\begin{remark}
Note that Conditions \ref{theorem:char_of_scmp_in_action_rep_cat} (a) and (c)
are equivalent for an object $X$ in a pointed finitely complete category.
\end{remark}

The content of the following proposition is well-known, we include a
proof to keep the paper more self contained.  

\begin{proposition} There are no (non-zero)
abelian \inj{} groups.  The group $\z{2}$ is \pinj{} but not \inj{}.
\end{proposition}
\begin{proof} Since every \inj{} group is \pinj{} and the only non-zero
abelian \pinj{} group is $\z{2}$ it suffices to prove the final claim.
However, this follows trivially from fact that the canonical monomorphism
from $\z{2}$ into $\z{4}$ is a normal monomorphism which isn't a split
monomorphism.
\end{proof}

\begin{remark}
Note that Baer's theorem which can stated using our terminology, as:
every \inj{} group is \scmp{}, is now a corollary of Theorem
\ref{theorem:char_of_scmp_in_action_rep_cat} via the previous
proposition. Furthermore, we have that if a group $G$ is \pinj{}, then
$G$ is \inj{} or $G\cong\z{2} \times H$ where 
$H$ is \inj{}.

\end{remark}

\section{Characteristic monomorphisms}
\label{section:characteristic monos}

The main purpose of this section is show that there is a common categorical
explanation behind why the derivation algebra of a perfect Lie algebra
with trivial center
and the automorphism group of a (characteristically) simple non-abelian group
are complete.

Recall that a subgroup $S$ of a group $X$ is called characteristic if every
automorphism of $X$ restricts to an automorphism of $S$. Recall also that
a subgroup $S$ of $X$ is characteristic if and only if whenever
$X$ is normal in $Y$, then $S$ is normal in $Y$. Generalizing this latter
condition A.\ S.\ Cigoli and A.\ Montoli introduced and studied the notion
of a characteristic subobjects in semi-abelian categories
\cite{CIGOLI_MONTOLI:2015}. Later D.\ Bourn gave a different definition in a
more general context \cite{BOURN:2014}, which coincides with the previously
mentioned one in the semi-abelian
context.  Here we say that a morphism $u:S\to X$ in
a category $\C$ is
characteristic monomorphism if for each Bourn-normal monomorphism $n: X\to Y$
the composite $un$ is a Bourn-normal monomorphism. In
\cite{CIGOLI_GRAY_VAN_DER_LINDEN:2015a} it was shown that for a semi-abelian
category, a morphism $u:S\to X$ is a characteristic monomorphism if and only if
for each protosplit monomorphism $\kappa : X\to A$ the composite $\kappa u$ is a
normal monomorphism. This fact essentially remains true in any pointed
\emph{Mal'tsev} category using the above definition of characteristic monomorphism.
Recall that a category is Mal'tsev if it is finitely complete 
and each (internal) reflexive relation
is an (internal) equivalence relation. Note that Mal'tsev categories were first
introduced and studied in \cite{CARBONI_LAMBECK_PEDICCHIO:1991} with exactness
as part of the definition, exactness was removed
in \cite{CARBONI_PEDICCHIO_PIRAVANO:1992}.
\begin{proposition} 
\label{proposition:equiv_char}
Let $\C$ be a pointed finitely complete Mal'tsev category
and let $u:S\to X$ be morphism in $\C$. The following are equivalent:
\begin{enumerate}[(a)]
\item $u$ is a characteristic monomorphism;
\item for
each protosplit normal monomorphism $\kappa : X\to A$
the composite $\kappa u$ is a Bourn-normal monomorphism;
\item for each reflexive relation $(R,r_1,r_2)$
on an object $Y$ with $k:X\to R$ as kernel of $r_1$, there exists a monomorphism
of reflexive relations $ v : (T, t_1,t_2) \to (R,r_1,r_2)$ with $l:S\to T$ as
kernel of $t_1$ such that $vl=ku$;
\item the same as (c) but replace ``reflexive relation'' by ``equivalence relation''.
\end{enumerate}
\end{proposition}
\begin{proof} Trivially (a) implies (b) and (c) is equivalent
to (d).  Now suppose that (b) holds and let $(R,r_1,r_2)$ be a reflexive
relation on $Y$ with $k:X\to R$ as kernel of $r_1$. Since $r_1$ is a split
epimorphism we know that $k$ is a protosplit monomorphism and hence $ku$ is a
Bourn-normal monomorphism by assumption. Let us write $(\bar T, \bar t_1,\bar
t_2)$ for the equivalence relation on $R$ that $k$ is zero-class of and
 $\bar l : S\to
\bar T$ for the kernel of $\bar t_1$ so that $\bar t_2 \bar l = ku$. Now,
suppose that $e: Y\to R$ and $\bar f: R\to \bar T$ are the unique morphisms such
that $r_i e=1_Y$ and $\bar t_i \bar f = 1_R$ and consider the pullback of split
extensions 
\[ \xymatrix@!@C=-1.5ex@R=-1.5ex{ S \ar[dr]^{1_S} \ar[rr]^{l} \ar[dd]_{u} && T
\ar@<0.5ex>[rr]^(0.6){t_1}\ar[dd]_(0.25){v}|(0.5){\hole}\ar[dr]^{w} && Y
\ar@<0.5ex>[ll]^(0.4){f}\ar[dr]^{e}\ar@{=}[dd]|(0.45){\hole}|(0.55){\hole}\\ &
S \ar[rr]^(0.3){\bar l}\ar[dd]_(0.3){ku} &&\bar T \ar@<0.5ex>[rr]^(0.3){\bar
t_1}\ar[dd]^(0.7){\langle \bar t_1,\bar t_2\rangle} && R
\ar@<0.5ex>[ll]^(0.7){\bar f}\ar@{=}[dd]\\ X
\ar[rr]_(0.3){k}|(0.5){\hole}\ar[dr]_{k} && R
\ar@<0.5ex>[rr]^(0.3){r_1}|(0.5){\hole}\ar[dr]_(0.3){\langle er_1,1\rangle} &&
Y\ar@<0.5ex>[ll]|(0.5){\hole}^(0.7){e} \ar[dr]^{e}\\ & R \ar[rr]_{\langle 0
,1\rangle} && R\times R \ar@<0.5ex>[rr]^{\pi_1} && R.\ar@<0.5ex>[ll]^{\langle
1,1\rangle} } \] Setting $t_2 = r_2v$ we obtain the desired monomorphism of
reflexive relations $v : (T,t_1,t_2) \to (R,r_2,r_2)$ proving (b) implies (c).
To complete the claim we will show that (d) implies (a). Suppose that (d) holds
and $n : X\to Y$ is a Bourn-normal monomorphism. By assumption there is an
equivalence relation $(R,r_1,r_2)$ on $Y$ with $k: X\to R$ as kernel of $r_1$
such that $n=r_2k$. According to (d) there is a monomorphism of equivalence
relations $v : (T,t_1,t_2) \to (R,r_1,r_2)$ with $l: S\to T$ as kernel of $t_1$
such that $vl=ku$. This means that $t_2l=r_2vl=r_2ku=nu$ proving that $nu$ is
Bourn-normal.  
\end{proof}
Recall that a pointed category with finite limits can be equivalently
defined to be strongly protomodular in the sense of D.{} Bourn
\cite{BOURN:2000} if it is protomodular and for each morphism of
split extensions
\[
\xymatrix{
X \ar[d]_{f} \ar[r]^{\kappa} & A \ar@<0.5ex>[r]^{\alpha}\ar[d]^{g} & B\ar@<0.5ex>[l]^{\beta}\ar@{=}[d]\\
Z \ar[r]_{\sigma} & C\ar@<0.5ex>[r]^{\gamma} & B\ar@<0.5ex>[l]^{\delta}
}
\]
the morphism $f$ is Bourn-normal if and only if $\sigma f$ is Bourn-normal.
Note that part (i) of the following theorem is essentially known, see Proposition 3.1
of \cite{BOURN:2014}.
\begin{theorem}
\label{theorem:char_of_char}
Let $\C$ be a pointed protomodular category,
let $X$ be an object in $\C$ such that the generic split extension with
kernel $X$ exists, and let $u: S\to X$ be a morphism.
\begin{enumerate}[(i)]
\item The morphism $u:S\to X$ is a characteristic
monomorphism if and only if
the composite $ku$ of $u$ and $k : X\to [X]\ltimes X$ is
a Bourn-normal monomorphism;
\item 
If $u:S\to X$ is a characteristic monomorphism, then
the generic split extension with kernel $(S,X,u)$ exists in $\C^\two$
and $q_2:[S,X,u]\to [X]$ is an isomorphism.
\end{enumerate}
Furthermore, when
$\C$ is strongly protomodular the converse of (ii) holds.
\end{theorem}
\begin{proof} Suppose $u:S\to X$ is a morphism in $\C$. By definition, if $u$
is a characteristic monomorphism,
then $ku: S\to [X]\ltimes X$ is Bourn-normal. Conversely, suppose
$ku: S\to [X]\ltimes X$ is Bourn-normal and $\kappa : X\to A$ is the kernel
of a split epimorphism $\alpha:A\to B$ with splitting $\beta : B\to A$.
Accordingly there is a unique morphism of split extensions
\begin{equation}
\label{diagram:unique_morphism_to_generic}
\vcenter{
\xymatrix{ 
X\ar@{=}[d] \ar[r]^-{\kappa} & A \ar[d]^{f} \ar@<0.5ex>[r]^-{\alpha} &
B\ar@<0.5ex>[l]^-{\beta}\ar[d]^{g}\\ X \ar[r]^-{k} & [X]\ltimes X \ar@<0.5ex>[r]^-{p_1} &
[X].\ar@<0.5ex>[l]^-{i} }
}
\end{equation}
By protomodularity the right and downward directed arrows of the right hand
square of \eqref{diagram:unique_morphism_to_generic} form a pullback and
hence the morphism $\langle f, \alpha \rangle : A\to ([X]\ltimes X)\times B$
is a monomorphism. Therefore, since
$\langle f,\alpha \rangle \kappa u = \langle ku, 0\rangle$ which is a
Bourn-normal monomorphism (being the intersection of the Bourn-normal monomorphisms
$1\times ku$ and $\langle 1,0\rangle$ - see e.g.{} Proposition 3.2.6 of
\cite{BORCEUX_BOURN:2007}) it follows that $\kappa u$ is a Bourn-normal
monomorphism as desired.
Now suppose $u$ is a characteristic monomorphism, $H=[X]\ltimes X$ 
and $(\bar T,\bar t_1, \bar t_2)$ is the equivalence relation that
$ku$ is the zero class of.
 Since $\C$ is protomodular there is a unique (up to isomorphism) equivalence
 relation to which
 $ku$ is normal. This means that $(ku, (\bar T,\bar t_1,\bar t_2))$
is the normalizer of $ku$ (in the sense of \cite{BOURN_GRAY:2015}) and hence
by Proposition 2.4 of \cite{BOURN_GRAY:2015} the front face of the pullback
of split extensions
\[ 
\xymatrix@!@C=-1.8ex@R=-1.8ex{
S \ar[dr]^{1_S} \ar[rr]^{l'} \ar[dd]_{u} && T'
\ar@<0.5ex>[rr]^(0.6){t'_1}\ar[dd]_(0.25){v'}|(0.5){\hole}\ar[dr]^{w'} && [X]
\ar@<0.5ex>[ll]^(0.4){f'}\ar[dr]^{i}
\ar@{=}[dd]|(0.45){\hole}|(0.55){\hole}\\
& S \ar[rr]^(0.3){\bar l}\ar[dd]_(0.3){ku} &&\bar T \ar@<0.5ex>[rr]^(0.3){\bar
t_1}\ar[dd]^(0.3){\langle \bar t_1,\bar t_2\rangle} && H
\ar@<0.5ex>[ll]^(0.7){\bar f}\ar@{=}[dd]\\
X \ar[rr]_(0.3){
k}|(0.5){\hole}\ar[dr]_{k} && H
\ar@<0.5ex>[rr]^(0.3){p_1}|(0.5){\hole}\ar[dr]_(0.4){\langle ip_1,1\rangle}
&&
[X]\ar@<0.5ex>[ll]|(0.5){\hole}^(0.7){i} \ar[dr]^{i}\\ & H \ar[rr]_{\langle 0
,1\rangle} && H\times H \ar@<0.5ex>[rr]^{\pi_1} && H\ar@<0.5ex>[ll]^{\langle
1,1\rangle}
}
\]
is a $K$-precartesian. However, by Lemma 2.7 of \cite{BOURN_GRAY:2015} this
means that the back face is also $K$-precartesian and hence by
Lemma 2.6 of \cite{GRAY:2017} is the generic split extension with kernel
$(S,X,u)$ in $\C^\two$.

To see that the final claim follows from strong protomodularity.
Just note that in the
diagram \eqref{generic_split_ext_in_mor_cat} with $f=u$ and $q_2$ an
isomorphism, strong protomodularity implies that the composite $kf$ is
Bourn-normal. 
\end{proof}
 
\begin{theorem}
\label{theorem:char_of_centerless_char}
Let $\C$ be a pointed
protomodular category, and let $X$ be an object in
$\C$ with trivial center such that the generic split extension with kernel $X$
exists. A morphism $u:S\to X$ is a characteristic monomorphism if and only if
the composite $\conj_Xu: S\to [X]$ is Bourn-normal.  
\end{theorem}
\begin{proof}
Since $X$ has trivial center, the morphism $\conj_X$ is monomorphism and hence
(by protomodularity)
so is
the middle morphism $\langle p_1,p_2\rangle$ in the morphism 
\[ \xymatrix@R=5ex{ X
\ar[r]^-{k}\ar[d]_{\conj_X} & [X]\ltimes X \ar@<0.5ex>[r]^-{p_1}\ar[d]^{\langle
p_1,p_2\rangle} & [X]\ar@<0.5ex>[l]^-{i}\ar@{=}[d]\\ [X] \ar[r]^-{\langle
0,1\rangle} & [X]\times [X] \ar@<0.5ex>[r]^-{\pi_1} & [X]
\ar@<0.5ex>[l]^-{\langle 1,1\rangle} } \] of split extensions. Since $\langle 0,
\conj_X u\rangle = \langle p_1,p_2\rangle k u$ is a Bourn-normal monomorphism and
$\langle p_1,p_2\rangle$ is a monomorphism, it follows that $ku$ is a
Bourn-normal monomorphism. The claim now follows by the previous proposition.
\end{proof}
%%%%%%%%%%
\begin{proposition}
\label{proposition: normal_subs_of_pinj_char}
Let $\C$ be a pointed protomodular category. Every normal
subobject of a \pinj{} object is characteristic.  
\end{proposition}
\begin{proof} Let $X$ be a \pinj{} object and let $n: S\to X$ be a Bourn-normal
monomorphism. If $\kappa : X\to A$ is a protosplit monomorphism, then according
to Lemma 
\ref{lemma:split_extensions_are_products}
 there is an isomorphism
$\psi: B\times X \to A$, where $B$ is the (object part of the) cokernel of
$\kappa$, such that $\kappa = \psi \langle 0,1\rangle$. Since $\langle
0,n\rangle$ is Bourn-normal it follows that $\kappa n = \psi \langle 0,1\rangle
n$ is Bourn-normal.  
\end{proof}
 
\begin{theorem}
\label{theorem:one_step}
Let $\C$ be a pointed protomodular category and let
$X$ be an object in $\C$ such that the generic split extension with kernel
$X$ exists.
There is a generic split extension with kernel $[X]$ and the conjugation morphism $\conj_X : X \to
[X]$ is a characteristic monomorphism if and only if 
$X$ has trivial center and $[X]$ is \scmp{}.
\end{theorem}
\begin{proof} Since (i) the morphism $\conj_X$ is a Bourn-normal monomorphism
if and only if $X$ has trivial center (for the ``if'' part use Proposition
\ref{proposition:X_centerless_c_X_centralizer_free}
 , and for the converse just recall that the center is the kernel of $c_X$); (ii) by Theorem
\ref{theorem:char_of_scmp_in_action_rep_cat}
if $[X]$ is \scmp{}, then there is a generic split extension with
kernel $[X]$, it follows that in addition to the assumptions above we may
assume that $X$
has trivial
center and that there is a generic split extension with kernel $[X]$, and
prove that $\conj_X$ is a characteristic monomorphism if and only if $[X]$ is
\scmp{}.
Suppose that $\conj_X$ is a characteristic monomorphism. Then according to
Theorem
\ref{theorem:char_of_char}
the generic split extension with kernel $(X,[X],\conj_X)$ exists in $\C^\two$. By
Proposition
\ref{proposition:X_centerless_[c_X]=c_[X]}
this means that the split extension classifier of
$(X,[X],\conj_X)$ is $([X],[[X]],\conj_{[X]})$ and so
by Theorem
\ref{theorem:char_of_char}
(again)
the morphism $\conj_{[[X]]}$ is an isomorphism. Therefore, $[X]$ is \scmp{} by
Theorem
\ref{theorem:char_of_scmp_in_action_rep_cat}.
The converse follows from Proposition
\ref{proposition: normal_subs_of_pinj_char}.
\end{proof}

We obtain the following known fact as a corollary:
\begin{proposition}
If a group $X$ is characteristically simple (i.e. it has no proper characteristic subobjects)
 and non-abelian, then $\Aut(X)$ is \inj{} (=\scmp{}).
\end{proposition}
\begin{proof}
Suppose that $X$ is a characteristically simple non-abelian group and suppose
 that $\theta$ is an automorphism
 of $\Aut(X)$. Since the center is always characteristic it follows that $Z(X)=0$
 and hence $c_X$ is a normal monomorphism.
 Forming the pullback
\[
\xymatrix@C=4.5ex@R=4ex{
K \ar[r]^{u}\ar[d]_{v} & X\ar[d]^{\conj_X}\\
X\ar[r]_-{\theta \conj_X} & \Aut(X)
}
\]
we see that $\conj_Xu$ is normal and hence by Theorem
\ref{theorem:char_of_centerless_char}
that $u$ is characteristic. By assumption this means that
either $K=0$ or $u$ is an isomorphism. However, since by Proposition
\ref{proposition:X_centerless_c_X_centralizer_free} $\conj_X$ has trivial
centralizer it follows that $K$ is not trivial and hence $u$ must be an isomorphism.
Essentially the same argument implies that $v$ is an isomorphism which
proves that $\conj_X$ is characteristic and hence by the previous theorem
implies that $\Aut(X)$ is complete.
\end{proof}

\begin{proposition} Let $\C$ be a category of interest in the sense of G. Orzech
 \cite{ORZECH:1972}
 such that the group operation (required to exists) is commutative,
and let $X$ be an object in $\C$.
\begin{enumerate}[(i)]
\item If $X$ is perfect, then every normal monomorphism with domain $X$ is a
characteristic monomorphism;  
\item if $X$ is perfect, has trivial center, and the generic split extension with
kernel $X$ exists, then $[X]$ is \scmp.
\end{enumerate} 
\end{proposition}
\begin{proof}
Recall that for such a variety of universal algebras a subobject $S \leq Y$
is normal if and only if for each $s$ in $S$, each $y$ in $Y$
and each binary operation $*$ (excluding addition), $s*y$ and $y*s$ are in
 $S$. Recall also that $[X,X]$ is the subalgebra of $X$ generated by
 elements of the form $x_1*x_2$ where $x_1$ and $x_2$ are elements of $X$
 and $*$ is a binary operation (exlcuding addition). Now suppose $X$ is perfect
 (i.e. $[X,X]=X$), $X$ is a normal
subobject of $Y$, and $Y$ is a normal subobject of $Z$.
 Since $\C$ is category of interest for binary operations $*$ and $\sqbullet$
 we known that for some postive integer $n$ there are binary operations $*_1,...,*_n$ and
 $\sqbullet_1,...,\sqbullet_n$ and a term $w$ such that (in particular) for
 all $x_1,x_2$ in
 $X$ and $z$ in $Z$ 
$(x_1*x_2)\sqbullet z = w(x_1*_1(x_2\sqbullet_1z),...,x_1*_m(x_2\sqbullet_m z),
x_2*_{m+1} (x_1\sqbullet_{m+1} z),...,x_2*_n(x_1\sqbullet_n z))$, where
$m$ is an integer between $0$ and $n$.
 Therefore, since each $x_j\sqbullet_i z$ is in $Y$ it follows that
 $x_k*_l(x_j\sqbullet z)$ is in $X$ and hence
 $(x_1*x_2)\sqbullet z$ is in $X$. A
similar calculation shows that $z\sqbullet (x_1*x_2)$ is in $X$.
Since $X$ is perfect we know that $X$ is generated by products and hence 
 via the previous calculations is normal in $Z$. This proves (i).
Combining (i) with Proposition
\ref{proposition:X_centerless_c_X_centralizer_free} and Theorem
\ref{theorem:one_step} we obtain (ii).
\end{proof}
\begin{remark}
Applying the previous proposition to the category of Lie algebras shows that
a perfect Lie algebra with trivial center has derivation algebra \inj{}.
\end{remark}
\providecommand{\bysame}{\leavevmode\hbox to3em{\hrulefill}\thinspace}
\providecommand{\MR}{\relax\ifhmode\unskip\space\fi MR }
% \MRhref is called by the amsart/book/proc definition of \MR.
\providecommand{\MRhref}[2]{%
  \href{http://www.ams.org/mathscinet-getitem?mr=#1}{#2}
}
\providecommand{\href}[2]{#2}

\end{document}